\documentclass[10pt]{article}

\usepackage{ssArxiv}

\usepackage{microtype}
\usepackage{graphicx}
\usepackage{subfigure}
\usepackage{booktabs}
\usepackage{hyperref}

\usepackage{color}
\usepackage{algorithm}
\usepackage{algorithmic}
\usepackage{natbib}
\usepackage{eso-pic}
\usepackage{forloop}

\usepackage{ams math,amsbsy,amsgen,amscd,amssymb,amsthm,amsfonts}
\usepackage{mathtools}
\usepackage{bm}
\usepackage{enumitem}
\usepackage{microtype}
\usepackage[nameinlink]{cleveref}
\newcommand{\norm}[1]{\|#1\|}
\newcommand{\R}{\mathbb{R}}
\newcommand{\dom}[1]{\mathrm{dom}(#1)}
\newcommand{\prox}{\mathrm{prox}}
\newcommand{\proj}{\mathrm{proj}}
\newcommand{\dist}{\mathrm{dist}}
\usepackage{url} 
\newcommand{\ip}[2]{\langle#1,#2\rangle}
\usepackage{setspace}
\newtheorem{lemma}{Lemma}
\newtheorem{corollary}{Corollary}
\newtheorem{remark}{Remark}
\newtheorem{theorem}{Theorem}
\newtheorem{obs}{Observation}
\newtheorem{definition}{Definition}
\newtheorem*{proof-sketch}{Proof Sketch}

\usepackage{graphics}

\makeatletter
  \def\title@font{\Large\bfseries}
  \let\ltx@maketitle\@maketitle
  \def\@maketitle{\bgroup%
    \let\ltx@title\@title%
    \def\@title{\resizebox{0.98\textwidth}{!}{%
      \mbox{\title@font\ltx@title}%
    }}%
    \ltx@maketitle%
  \egroup}
\makeatother
\title{Three Operator Splitting with a Nonconvex Loss Function}

 \author{\name Alp Yurtsever \email{alpy@mit.edu}\\
   \name Varun Mangalick \email{varunm22@gmail.com}\\
   \name Suvrit Sra \email{suvrit@mit.edu}\\
      \addr{Massachusetts Institute of Technology}
 }

\usepackage[title]{appendix}

\begin{document}

\maketitle

\begin{abstract}
We consider the problem of minimizing the sum of three functions, one of which is nonconvex but differentiable, and the other two are convex but possibly nondifferentiable. 
We investigate the Three Operator Splitting method (TOS) of \citet{davis2017three} with an aim to extend its theoretical guarantees for this nonconvex problem template. 
In particular, we prove convergence of TOS with nonasymptotic bounds on its nonstationarity and infeasibility errors. 
In contrast with the existing work on nonconvex TOS, our guarantees do not require additional smoothness assumptions on the terms comprising the objective; hence they cover instances of particular interest where the nondifferentiable terms are indicator functions. 
We also extend our results to a stochastic setting where we have access only to an unbiased estimator of the gradient. Finally, we illustrate the effectiveness of the proposed method through numerical experiments on quadratic assignment problems.

\end{abstract}

\section{Introduction}
\label{sec:introduction}

We study nonconvex optimization problems of the form:
\begin{equation}
\label{eqn:model-problem}
\min_{x \in \R^n} \quad \phi(x) := f(x) + g(x) + h(x),
\end{equation}
where $f :\R^n \to \R $ is continuously differentiable and potentially nonconvex, whereas $g$ and $h : \R^n \to \R \cup \{+\infty\}$ are proper lower-semicontinuous convex functions (potentially nonsmooth). Further, we assume that the domain of $g$, that is, $\dom{g} = \{x\in \R^n : g(x) < +\infty\}$, is bounded. 

Template \eqref{eqn:model-problem} enjoys a rich number of applications in optimization, machine learning, and statistics. Nonconvex losses arise naturally in several maximum likelihood estimation~\citep{mclachlan96} and M-estimation problems~\citep{ollila2014regularized,maronna2019robust}, in problems with a matrix factorization structure~\citep{zass2007nonnegative}, in certain transport and assignment problems~\citep{koopmans1957assignment,peyre2019computational}, among countless others. The nonsmooth terms in~\eqref{eqn:model-problem} can be used as regularizers, \emph{e.g.}, to promote joint behavior such as sparsity and low-rank~\citep{richard2012estimation}. Moreover, we can also split a complex regularizer into simpler terms for computational advantages, \emph{e.g.}, in group lasso with overlaps~\citep{jacob2009group}, structured sparsity~\citep{el2015totally}, or total variation \citep{barbero2018modular}. 

We obtain an important special case by choosing the nonsmooth terms $g$ and $h$ in \eqref{eqn:model-problem} as indicator functions of closed and convex sets $\mathcal{G}$ and $\mathcal{H} \subseteq \R^n$. In this case, \eqref{eqn:model-problem} turns into
\begin{equation}
\label{eqn:model-problem-constraint}
\min_{x \in \R^n} \quad f(x) \quad \text{subject to}\ 
\quad x \in \mathcal{G} \cap \mathcal{H}.
\end{equation}
We are particularly interested in the setting where $\cal G$ and $\cal H$ are \emph{simple} in the sense that we can project onto these sets efficiently, but not so easily onto their intersection. Some examples include learning with correlation matrices~\citep{higham2016anderson}, power assignment in wireless networks \citep{de2010traveling}, graph transduction \citep{shivanna2015spectral}, graph matching \citep{zaslavskiy2008path}, and quadratic assignment \citep{koopmans1957assignment,loiola2007survey}.

An effective way to solve~\eqref{eqn:model-problem} for convex $f$ with Lipschitz gradients is the Three Operator Splitting (TOS) method~\citep{davis2017three}, whose convergence has been well-studied (see \S\ref{sec:related}). But for nonconvex $f$, convergence properties of TOS are less understood (again, see~\S\ref{sec:related}). This gap motivates us to develop nonasymptotic convergence guarantees for TOS. Beyond theoretical progress, we highlight   the potential empirical value of TOS by evaluating it on a challenging nonconvex problem, the quadratic assignment problem (QAP).

\textbf{Contributions.} We summarize our contributions towards the convergence analysis of nonconvex TOS below. \begin{itemize}[itemsep=0.5em,topsep=0pt,parsep=0pt,partopsep=0pt,leftmargin=*]
\item[$\triangleright$] We first discuss how to quantify convergence of TOS to first-order stationary points for both templates~\eqref{eqn:model-problem} and \eqref{eqn:model-problem-constraint}. Specifically, we propose to measure approximate stationarity based on a variational inequality. Thereafter, we prove that the associated non-stationarity error is smaller than $\epsilon$ (in expectation over a random iteration counter) after $T = \mathcal{O}(1/\epsilon^3)$ iterations (and gradient evaluations).
\item[$\triangleright$] We extend our analysis to stochastic optimization where we have access only to an unbiased estimate of the gradient $\nabla f$. In this case, we prove that the error is smaller than $\epsilon$ (in expectation) after $T = \mathcal{O}(1/\epsilon^3)$ iterations. The corresponding algorithm requires drawing $\mathcal{O}(1/\epsilon^5)$ \emph{i.i.d.}\ stochastic gradients.
\end{itemize}
Finally, we evaluate TOS on the quadratic assignment problem using the well-known QAPLIB benchmark library~\citep{burkard1997qaplib}. Remarkably, TOS performs significantly better than the theory suggests: we find that it converges locally linearly. Understanding this behavior could be a potentially valuable question for future study.

 \subsection{Related Works}
\label{sec:related}
\citet{davis2017three} introduce TOS for solving the monotone inclusion of three operators, one of which is co-coercive. TOS gives us a simple algorithm for~\eqref{eqn:model-problem} when $f$ is smooth and convex, since the gradient of a smooth convex function is co-coercive. At each iteration, TOS evaluates the gradient of $f$ and the proximal operators of $g$ and $h$ once, separately. 
TOS extends various previous operator splitting schemes such as the forward-backward splitting, Douglas-Rachford splitting, Forward-Douglas-Rachford splitting~\citep{briceno2015forward}, and the Generalized Forward-Backward splitting~\citep{raguet2013generalized}. 

The original algorithm of \citet{davis2017three} requires knowledge of the smoothness constant of $f$; \citet{pedregosa2018adaptive} introduce a variant of TOS with backtracking line-search that bypasses this restriction. \citet{zong2018convergence} analyze convergence of TOS with inexact oracles where both the gradient and proximity oracles can be noisy. 

Existing work on TOS applied to nonconvex problems limits itself to the setting where at least two terms in~ \eqref{eqn:model-problem} have Lipschitz continuous gradients. Under this assumption, \citet{liu2019envelope} identify an envelope function for TOS, which permits one to interpret TOS as gradient descent for this envelope  under a variable metric. Their envelope generalizes the well-known Moreau envelope as well as the envelopes for Douglas-Rachford and Forward-Backward splitting introduced in~\citep{patrinos2014douglas} and \citep{themelis2018forward}. 

\citet{bian2020three} present convergence theory for TOS under the same smoothness assumptions. They show that the sequence generated by TOS with a carefully chosen step-size converges to a stationary point of \eqref{eqn:model-problem}. They also prove asymptotic convergence rates under the assumption that the Kurdyka-\L{}ojasiewicz property holds (see Definition~2.3 in \citep{bian2020three}). 

Our focus is significantly different from these prior works on nonconvex TOS. In contrast to the settings of \citep{liu2019envelope} and \citep{bian2020three}, we do \emph{not} impose any assumption on the smoothness of $g$ and $h$. However, we do assume that the nonsmooth terms $g$ and $h$ are convex and the problem domain is bounded. 

In particular, our setting includes nonconvex minimization over the intersection of two simple convex sets, which covers important applications such as the quadratic assignment problem and graph matching. Note that these problems are challenging for TOS even in the convex setting, because the intermediate estimates of TOS can be infeasible and the known guarantees on the convergence rate of TOS fail, see the discussion in Section~3.2 in \citep{pedregosa2018adaptive}.

Finally, \citet{yurtsever2016stochastic}, \citet{cevher2018stochastic}, \citet{zhao2018stochastic}, and \citet{pedregosa2019proximal} propose and analyze stochastic variants of TOS and related methods in the convex setting. We are unaware of any prior work on nonconvex stochastic TOS. 

\vspace{0.5em}

\textbf{Notation.} Before moving onto the theoretical development, let us summarize here key notation used throughout the paper. We use $\ip{\cdot}{\cdot}$ to denote the standard Euclidean inner product associated with the norm $\norm{\cdot}$. The distance between a point $x \in \R^n$ and a set $\mathcal{G} \subseteq \R^n$ is defined as $\dist(x,\mathcal{G}) := \inf_{y \in \mathcal{G}} \norm{x - y}$; the projection of $x$ onto $\mathcal{G}$ is given by $\proj_{\mathcal{G}}(x) := \arg\min_{y \in \mathcal{G}} \norm{x - y}$. We denote the indicator function of $\mathcal{G}$ by $\iota_{\mathcal{G}}: \R^n \to \{0,+\infty\}$, that takes $0$ for any $x \in \mathcal{G}$ and $+\infty$ otherwise. 
The proximal operator (or prox-operator) of a function $g: \R^n \to \R \cup \{+\infty\}$ is defined by $\prox_{g}(x) := \arg\min_{y \in \R^n} \{ g(y) + \frac{1}{2} \norm{x - y}^2 \}$. 
Recall that the prox-operator for the indicator function is the projection, \emph{i.e.}, $\prox_{\iota_{\mathcal{G}}}(x) = \proj_{\mathcal{G}}(x)$.

\section{Basic Setup: Approximate Stationarity}
\label{sec:approximate-stationarity}
We begin our analysis by setting up the notion of approximate stationarity that we will use to judge convergence.  For unconstrained minimization of smooth functions, gradient norm is a widely used standard measure. But the gradient norm is unsuitable in our case because of the presence of constraints and nonsmooth terms in the cost. 

Related work on operator splitting for nonconvex optimization typically considers the norm of a proximal gradient, or uses some other auxiliary differentiable function that converges to zero as we get closer to a first-order stationary point. See, for instance, the \emph{envelope functions} introduced by \citet{patrinos2014douglas}, \citet{themelis2018forward} and \citet{liu2019envelope}, or the \textit{energy function} defined by \citet{bian2020three}. However, these functions can characterize stationary points of \eqref{eqn:model-problem} only under additional smoothness assumptions on $g$ and $h$. They fail to capture important applications where both $g$ and $h$ are nonsmooth. 

In contrast, we consider a simple measure based on the variational inequality characterization of first-order stationarity.

\begin{definition}[Stationary point]
$\bar{z} \in \mathrm{dom}(\phi)$ is a first-order stationary point of \eqref{eqn:model-problem} if,~ for all $x \in \mathrm{dom}(\phi)$,
\begin{align}
\label{eqn:stationarity}
    \ip{\nabla f(\bar{z})}{\bar{z} - x} + g(\bar{z}) - g(x) + h(\bar{z}) - h(x) \leq 0.
\end{align}
See \Cref{lem:stationary-point} in the supplementary material for the technical details on condition \eqref{eqn:stationarity}. 
\end{definition}

We consider a perturbation of the bound in \eqref{eqn:stationarity} to define an approximately stationary point.

\begin{definition}[$\epsilon$-stationary point]
\label{def:epsilon-stationary}
We say $\bar{z} \in \mathrm{dom}(\phi)$ is an $\epsilon$-stationary point of \eqref{eqn:model-problem} if,~ for all $x \in \mathrm{dom}(\phi)$,
\begin{align}
\label{eqn:epsilon-stationarity}
    \ip{\nabla f(\bar{z})}{\bar{z} - x} + g(\bar{z}) - g(x) + h(\bar{z}) - h(x) \leq \epsilon.
\end{align}
\end{definition}
This is a natural extension of the notion of suboptimal solutions in terms of function values used in convex optimization. 
Similar measures for stationarity appear in the literature for various problems; see \emph{e.g.},~\citep{he2015non,nouiehed2019solving,malitsky2019golden,song2020optimistic}.

TOS is particularly advantageous for \eqref{eqn:model-problem} when the proximal operators of $g$ and $h$ are easy to evaluate separately but the proximal operator of their sum is difficult. For \eqref{eqn:model-problem-constraint}, this corresponds to optimization over $\mathcal{G} \cap \mathcal{H}$ by using only projections onto the individual sets and not onto their intersection. In this setting, we can achieve a feasible solution only in an asymptotic sense. Finding a \emph{feasible} $\epsilon$-stationary solution is an unrealistic goal. Thus, for \eqref{eqn:model-problem-constraint}, we consider a relaxation of \Cref{def:epsilon-stationary} that permits approximately feasible solutions. 

\begin{definition}[$\omega$-feasible $\epsilon$-stationary point]
\label{def:epsilon-feasible-stationary}
We say $\bar{z} \in \mathcal{G}$ is an $\omega$-feasible $\epsilon$-stationary point of \eqref{eqn:model-problem-constraint} if
\begin{gather}
\dist(\bar{z},\mathcal{H}) \leq \omega, \quad  \text{and} \label{eqn:infeasible-closeness} \\[0.25em]
        \ip{\nabla f(\bar{z})}{\bar{z} - x} \leq \epsilon, \quad \forall x \in \mathcal{G} \cap \mathcal{H}.\label{eqn:infeasible-stationarity}
\end{gather}

\end{definition}
\begin{remark}
For simplicity, we measure infeasibility of $\bar{z}$ via $\dist(\bar{z},\mathcal{H})$. This is suitable because the estimates of TOS remain in $\mathcal{G}$ by definition. We can also consider a slightly stronger notion of approximate feasibility given by $\dist(\bar{z},\mathcal{G} \cap \mathcal{H})$. 
However, this requires additional regularity conditions on $\mathcal{G}$ and $\mathcal{H}$ to avoid pathological examples. See, for instance, Lemma~1 in \citep{hoffmann1992distance} or Definition~2 in \citep{kundu2018convex}.
\end{remark}

The directional derivative condition  \eqref{eqn:infeasible-stationarity} is often used in the analysis of conditional gradient methods, and it is known as the \emph{Frank-Wolfe gap} in this literature. See \citep{jaggi2013revisiting,lacoste2016convergence,reddi2016stochastic,yurtsever2019conditional} for some examples. 

Approximately feasible solutions are widely considered in the analysis of primal-dual methods (but usually in the convex setting), see \citep{yurtsever2018conditional,kundu2018convex} and the references therein.
Remark that TOS can also be viewed as a primal-dual method \citep{pedregosa2018adaptive}. 

Problem~\eqref{eqn:model-problem-constraint} is challenging for TOS because of the infeasibility of the intermediate estimates, even when $f$ is convex. \citet{davis2017three} avoid this issue by evaluating the terms $h$ and $(f + g)$ at two different points, $x \in \mathcal{H}$ and $z\in \mathcal{G}$. 

However, $\bigl(f(z) + g(z)\bigr) + h(x)$ can be equal to the optimal objective value even when neither $x$ nor $z$ is close to a solution. 
We can address this issue by introducing a condition on the distance between $x$ and $z$. 
The following definition of an $\alpha$-close and $\beta$-stationary pair of points is crucial for our analysis. 

\begin{definition}[$\alpha$-close $\beta$-stationary pair]
\label{def:close-stationary-pair}
We say that $(\bar{x},\bar{z}) \in \dom{h} \times \dom{g}$ are $\alpha$-close and $\beta$-stationary points of \eqref{eqn:model-problem} if,~ for all $x \in \dom{\phi}$,
\begin{gather}
    \norm{\bar{z} - \bar{x}} \leq \alpha, \label{eqn:alpha-closeness}
\qquad \text{and} \\[0.5em]
    \ip{\nabla f(\bar{z})}{\bar{x} - x} + g(\bar{z}) - g(x) + h(\bar{x}) - h(x) \leq \beta. \label{eqn:beta-stationarity}
\end{gather}
\end{definition}
$\alpha$-close $\beta$-stationary points $(\bar{x},\bar{z})$ yield approximate solutions to \eqref{eqn:model-problem} and \eqref{eqn:model-problem-constraint} under appropriate assumptions. 

\begin{obs}\label{rmk:close-stationary-pair}
\textbf{(i).} Let $h$ be Lipschitz continuous on $\R^n$ with constant $L_h$. Assume that $\norm{\nabla f(z)}$ is bounded by $G_f$ for all $x \in \dom{g}$. Suppose that the points $(\bar{x},\bar{z})$ are $\alpha$-close and $\beta$-stationary. Then, $\bar{z}$ is an $\epsilon$-stationary point with $\epsilon = \alpha(G_f+L_h) + \beta$ as per \Cref{def:epsilon-stationary}. 

\textbf{(ii).} Let $g$ and $h$ be indicators of closed convex sets $\mathcal{G}$ and $\mathcal{H}$ respectively. Assume that $\norm{\nabla f(z)}$ is bounded by $G_f$ for all $x \in \mathcal{G}$. Suppose that the points $(\bar{x},\bar{z}) \in \mathcal{H} \times \mathcal{G}$ are $\alpha$-close and $\beta$-stationary. Then, $\bar{z}$ is an $\alpha$-feasible $\epsilon$-stationary point with $\epsilon = \alpha G_f + \beta$ as per \Cref{def:epsilon-feasible-stationary}. \end{obs}

\begin{proof}
\textbf{\textit{(i).}} Since $h$ is Lipschitz, we have 
\begin{align}\label{eqn:obs-1}
    h(\bar{x}) - h(x) \geq  h(\bar{z}) - h(x) - L_h \norm{\bar{z} - \bar{x}}.
\end{align}
And since $\norm{\nabla f(z)}$ is bounded, we have
\begin{align}\label{eqn:obs-2}
    \ip{\nabla f(\bar{z})}{\bar{x} - x} \geq \ip{\nabla f(\bar{z})}{\bar{z} - x} - G_f \norm{\bar{z} - \bar{x}}.
\end{align}
We get \eqref{eqn:epsilon-stationarity} with  $\epsilon = \alpha(G_f+L_h) + \beta$ by using \eqref{eqn:obs-1} and \eqref{eqn:obs-2} in \eqref{eqn:beta-stationarity} and bounding $\norm{\bar{z} - \bar{x}}$ by \eqref{eqn:alpha-closeness}.

\textbf{\textit{(ii).}} We get \eqref{eqn:infeasible-closeness} with $\omega = \alpha$ since
\begin{align}
    \dist(\bar{z},\mathcal{H}) = \inf_{x \in \mathcal{H}} \norm{\bar{z}- x} \leq \norm{\bar{z}-\bar{x}} .
\end{align}
$h(\bar{x}) = g(\bar{z}) = h(x) = g(x) = 0$ since $\bar{x} \in \mathcal{H}$, $\bar{z} \in \mathcal{G}$, and $x \in \mathcal{G} \cap \mathcal{H}$. Then, \eqref{eqn:infeasible-stationarity} follows from \eqref{eqn:beta-stationarity} by using \eqref{eqn:obs-2}. 
\end{proof}

\vspace{-0.5em}
We are now ready to present and analyze the algorithm.

\section{TOS with a Nonconvex Loss Function}
\label{sec:algorithm}

This section establishes convergence guarantees of TOS for solving Problems \eqref{eqn:model-problem} and \eqref{eqn:model-problem-constraint}. The method is detailed in \Cref{alg:three-operator-splitting}.

\begin{algorithm}[tb]
\setstretch{1.15}
   \caption{Three Operator Splitting (TOS)}
   \label{alg:three-operator-splitting}
\begin{algorithmic}
\vspace{0.25em}
   \STATE {\bfseries Input:} Initial point $y_1 \in \mathbb{R}^n$, step-size sequence $\{\gamma_t\}_{t=1}^T$ \\[0.25em]
   \FOR{$t=1,2,\ldots,T$}
   \STATE $z_{t} = \prox_{\gamma_t g} (y_t)$
   \STATE $x_{t} = \prox_{\gamma_t h} (2z_t - y_t - \gamma_t \nabla f(z_t))$
\STATE $y_{t+1} = y_t - z_t + x_t $
   \ENDFOR\\[0.25em]
   \STATE {\bfseries Return:} Draw $\tau$ uniformly at random from $\{1,2,\ldots, T\}$ and output ${z}_\tau$.
   \vspace{0.25em}
\end{algorithmic}
\end{algorithm}

\begin{theorem} \label{thm:TOS-Lipschitz}
Consider Problem~\eqref{eqn:model-problem} under the following assumptions:\\
\emph{(i)}~The domain of $g$ has finite diameter $D_g$,
\begin{equation*}
    \norm{x-y} \leq D_g, \quad \forall x, y \in \dom{g}.
\end{equation*}
\emph{(ii)}~$g$ is $L_g$-Lipschitz continuous on its domain, 
\begin{equation*}
    g(x) - g(y) \leq L_g \norm{x-y}, \quad \forall x, y \in \dom{g}.  
\end{equation*}
\emph{(iii)}~The gradient of $f$ is bounded by $G_f$ on the domain of $g$,
\begin{equation*}
    \norm{\nabla f(x)} \leq G_f, \quad \forall x \in \dom{g}.
\end{equation*}
\emph{(iv)}~$h$ is $L_h$-Lipschitz continuous on $\R^n$,
\begin{equation*}
    h(x) - h(y) \leq L_h \norm{x-y}, \quad \forall x, y \in \R^n.
\end{equation*}
Choose $y_1 \in \dom{g}$. 
Then, $z_\tau$ returned by TOS (\Cref{alg:three-operator-splitting}) after $T$ iterations with the fixed step-size $\gamma_t = \gamma = \frac{D_g}{2 (G_f + L_g + L_h) T^{2/3}}$ satisfies, for all $x$ in $\dom{g}$,
\begin{align}
 \label{eqn:theorem-TOS-Lipschitz}
    \mathbb{E}_{\tau}&[\ip{\nabla f (z_\tau)}{z_\tau - x} + g(z_\tau) - g(x) +  h(z_\tau) - h(x)] \leq \frac{4D_{g}(G_f + L_g + L_h)}{T^{1/3}}.
\end{align}
\end{theorem}

\begin{proof}[Proof sketch]
We start by writing the optimality conditions for the proximal steps on $x_t$ and $z_t$. 
Through algebraic modifications, we show that, for all $x \in \dom{\phi}$, 
\begin{align} 
\label{eqn:proof-sketch-step-1}
\begin{aligned}
\ip{\nabla f (z_t)}{x_t - x} + g(z_t) - g(x) + h(x_t) - h(x)  \leq \frac{1}{2\gamma} \Big( \norm{y_t - x}^2 - \norm{y_{t+1} - x}^2 - \norm{x_t - z_t}^2 \Big).
\end{aligned}
\end{align}
We take the average of this inequality over $t = 1,2,\ldots T$. The inverted terms with $y_t$ and $y_{t+1}$ cancel out since the step-size is fixed. As a result, we know that $(x_\tau,z_\tau)$ satisfy \eqref{eqn:beta-stationarity} with $\beta = {D_g^2} / ({2 \gamma T})$ in expectation. 

We also need to show that $(x_\tau,z_\tau)$ satisfy the proximity condition \eqref{eqn:alpha-closeness}.
To this end, we extend \eqref{eqn:proof-sketch-step-1} as 
\begin{align} \label{eqn:proof-sketch-step-2}
\begin{aligned}
- (G_f + L_h)\norm{x_t - z_t} - (G_f + L_g + L_h) D_g \leq \frac{1}{2\gamma} \Big( \norm{y_t - x}^2 - \norm{y_{t+1} - x}^2 - \norm{x_t - z_t}^2 \Big),
\end{aligned}
\end{align}
by using the boundedness of the domain \emph{(i)}, boundedness of the gradient norm  \emph{(iii)}, and Lipschitz continuity of $g$ and~$h$  \emph{(ii, \,iv)}. Again, we take the average over $t$ and eliminate the inverted terms. This leads to a second order inequality of $\mathbb{E}_\tau [\norm{z_\tau - x_\tau}]$. By solving this inequality, we get an upper bound on $\mathbb{E}_\tau [\norm{z_\tau - x_\tau}]$ in terms of the problem constants $G_f,L_g,L_h,D_g$, total number of iterations $T$, and step-size $\gamma$. By choosing $\gamma$ carefully, we ensure that $(x_\tau,z_\tau)$ are close and approximately stationary as per \Cref{def:close-stationary-pair}. We complete the proof by using \Cref{rmk:close-stationary-pair} \emph{(i)}. 
\end{proof}

Our proof is a nontrivial extension of the convergence guarantees of TOS to the nonconvex problems. The prior analysis for the convex setting is based on a fixed point characterization of TOS and on Fej\'{e}r monotonicity of $\norm{y_t - y_\star}$, where $y_\star$ denotes the fixed point of TOS, see Proposition~2.1 in \citep{davis2017three}. Unfortunately. this desirable feature is lost when we drop the convexity of $f$. 
Our approach of proving proximity between $x_\tau$ and $z_\tau$ via second-order inequality \eqref{eqn:proof-sketch-step-2} is nonstandard.

\begin{remark} \label{rmk:TOS-Lipschitz}
We highlight several points about \Cref{thm:TOS-Lipschitz}: \\[0.25em]
\emph{1.}~When $D_g, G_f, L_g,$ or $L_h$ is not known, one can use $\gamma_t = \frac{\gamma_0}{T^{2/3}}$ for any $\gamma_0 > 0$. The convergence rate in \eqref{eqn:theorem-TOS-Lipschitz} still holds but with different constants. We chose the specific step-size in \Cref{thm:TOS-Lipschitz} in order to simplify the bounds. \\[0.25em]
\emph{2.}~Assumption \emph{(iii)} holds automatically if $f$ is smooth since $\dom{g}$ is bounded. \\[0.25em]
\emph{3.}~We can relax assumption \emph{(iv)} as follows: $h$ is $L_h$-Lipschitz continuous on $\dom{h}$, and $\dom{h} \supseteq \dom{g}$. \\[0.25em]
\emph{4.}~We can slightly tighten the constants in \eqref{eqn:theorem-TOS-Lipschitz}. We defer the details to the supplementary material. \\[0.25em]
\emph{5.}~Our guarantees hold in expectation for the estimation at a randomly drawn iteration. This is a common technique in the nonconvex analysis. For example, see \citep{reddi2016stochastic,reddi2016stochastic-variance,yurtsever2019conditional} and the references therein.\end{remark}

\begin{corollary} \label{thm:TOS-indicator-Lipschitz}
Consider Problem \eqref{eqn:model-problem} under the following assumptions:\\[0.25em]
\emph{(i)}~$g$ is the indicator function of a convex closed bounded set $\mathcal{G} \subseteq \R^n$ with a finite diameter $D_\mathcal{G} := \sup_{x,y \in \mathcal{G}} \norm{x-y}$. \\[0.25em]
\emph{(ii)}~$\nabla f$ is bounded on $\mathcal{G}$, \emph{i.e.}, $\norm{\nabla f(x)} \leq G_f, ~ \forall x \in \mathcal{G}$. \\[0.25em]
\emph{(iii)}~$h$ is $L_h$-Lipschitz continuous on $\R^n$.\\[0.5em]
Choose $y_1 \in \mathcal{G}$.
Then, $z_\tau$ returned by TOS after $T$ iterations with the fixed step-size $\gamma_t = \frac{D_\mathcal{G}}{2 (G_f + L_h) T^{2/3}}$ satisfies
\begin{equation}\label{eqn:eps-solution-corollary}
\begin{aligned}
\mathbb{E}_{\tau}[\ip{\nabla f (z_\tau)}{z_\tau - & x} +  h(z_\tau) - h(x)] \leq \frac{4D_{\mathcal{G}}(G_f + L_h)}{T^{1/3}},
\quad \forall x \in \mathcal{G}.
\end{aligned}
\end{equation}
\end{corollary}

\begin{proof} \Cref{thm:TOS-indicator-Lipschitz} follows from  \Cref{thm:TOS-Lipschitz} with 
 $\dom{g} = \dom{\phi} = \mathcal{G}$. Assumptions \emph{(i)} and \emph{(ii)} in \Cref{thm:TOS-Lipschitz} hold with $D_g = D_\mathcal{G}$ and $L_g = 0$. We have $g(z_\tau) = g(x) = 0$ because $z_\tau$ and $x$ belong to $\mathcal{G}$.
\end{proof}

\begin{remark}
The $\epsilon$-approximate solution (in expectation) that we consider in \eqref{eqn:theorem-TOS-Lipschitz} and \eqref{eqn:eps-solution-corollary} reminds the Frank-Wolfe gap (in expectation) used in \citep{reddi2016stochastic,yurtsever2019conditional}. When $h$ is missing and $g$ is the indicator function, the Frank-Wolfe gap quantifies the error by 
$\mathbb{E}_{\tau} \big[ \max_{x \in \mathcal{G}} \, \ip{\nabla f(z_\tau)}{z_\tau - x} \big]$. \eqref{eqn:eps-solution-corollary} holds for all $x \in \mathcal{G}$ so we can take the maximum over $x$ and get the bound on $\max_{x \in \mathcal{G}} \mathbb{E}_{\tau}[\ip{\nabla f (z_\tau)}{z_\tau - x} +  h(z_\tau) - h(x)]$. Note that $\max_{x \in \mathcal{G}} \mathbb{E}_\tau [\,\cdot\,] \leq  \mathbb{E}_\tau [\max_{x \in \mathcal{G}} (\,\cdot\,)]$. We leave the question whether similar guarantees hold for $ \mathbb{E}_\tau [\max_{x \in \mathcal{G}} (\,\cdot\,)]$ open. 
\end{remark}

\Cref{thm:TOS-Lipschitz} does not apply to Problem \eqref{eqn:model-problem-constraint} because indicator functions fail Lipschitz continuity assumption \textit{(iv)} in \Cref{thm:TOS-Lipschitz}. 
The next theorem establishes convergence guarantees of TOS for Problem~\eqref{eqn:model-problem-constraint}.

\begin{theorem}
\label{thm:TOS-indicator}
Consider Problem~\eqref{eqn:model-problem-constraint} under the following assumptions:\\[0.5em]
\emph{(i)}~$\mathcal{G} \subseteq \R^n$ is a bounded closed convex set with a finite diameter $D_\mathcal{G} := \sup_{x,y \in \mathcal{G}} \norm{x-y}$. \\[0.5em]
\emph{(ii)}~$\nabla f$ is bounded on $\mathcal{G}$, \emph{i.e.}, $\norm{\nabla f(x)} \leq G_f, ~ \forall x \in \mathcal{G}$. \\[0.5em]
\emph{(iii)}~$\mathcal{H} \subseteq \R^n$ is a closed convex set. \\[0.5em]
Then, $z_\tau$ returned by TOS (\Cref{alg:three-operator-splitting}) after $T$ iterations with the fixed step-size $\gamma_t = \frac{D_\mathcal{G}}{2 G_f T^{2/3}}$ satisfies
\begin{align}
\mathbb{E}_\tau [\dist(z_\tau, \mathcal{H})] & \leq \frac{3D_{\mathcal{G}}}{T^{1/3}}, \\[0.25em]
\mathbb{E}_\tau [\ip{\nabla f (z_\tau)}{z_\tau - x}] & \leq  \frac{4G_fD_{\mathcal{G}}}{T^{1/3}}, \quad \forall x \in \mathcal{G} \cap \mathcal{H}.\notag
\end{align}

\begin{proof}[Proof sketch]
The analysis is similar to the proof of \Cref{thm:TOS-Lipschitz}. 
We use \Cref{rmk:close-stationary-pair} \emph{(ii)} once we show that $(x_\tau,z_\tau)$ are close and approximately stationary.
\end{proof}

\end{theorem}

\subsection{Extensions for More Than Three Functions}
\label{sec:extensions-for-more-than-three-terms}

Consider the extension of Problem \eqref{eqn:model-problem} with an arbitrary number of nonsmooth terms (equivalently, an extension of Problem \eqref{eqn:model-problem-constraint} with an arbitrary number of constraints):
\begin{equation}
\label{eqn:model-problem-arbitrary-terms}
\min_{x \in \R^n} \quad f(x) + \sum_{i=1}^m g_i(x) .
\end{equation}

One can solve this problem with TOS via a product-space formulation (see Section~6.1 in \citep{briceno2015forward}). 
We introduce slack variables $x^{(0)}, x^{(1)}, \ldots, x^{(m)} \in \R^n$ and reformulate Problem~\eqref{eqn:model-problem-arbitrary-terms} as
\begin{equation}
\label{eqn:model-problem-product-space}
\begin{aligned}
& \min_{x^{(i)} \in \R^n} & & f(x^{(0)}) + \sum_{i=1}^m g_i(x^{(i)}) \\
& \mathrm{subj.~to} & & x^{(0)} = x^{(1)} = \ldots = x^{(m)}.
\end{aligned}
\end{equation}
Problem \eqref{eqn:model-problem-product-space} is an instance of Problem~\eqref{eqn:model-problem} in $\R^{(m+1)n}$. 
We can use TOS for solving this problem. 
\Cref{alg:three-operator-splitting-arbitrary-terms} in the supplementary material describes the algorithm steps.

\section{Stochastic Nonconvex TOS}

In this section, the differentiable term is the expectation of a function of a random variable, \emph{i.e.}, $f(x) = \mathbb{E}_\xi \tilde{f}(x,\xi)$, where $\xi$ is a random variable with distribution $\mathcal{P}$:
\begin{equation} \label{eqn:model-problem-stochastic}
\min_{x \in \R^n} \quad \phi(x) := \mathbb{E}_{\xi} \tilde{f}(x,\xi) + g(x) + h(x).
\end{equation}
This template covers a large number of applications in machine learning and statistics, including the finite-sum formulations that arise in M-estimation and empirical risk minimization problems.

In this setting, we replace $\nabla f(z_t)$ in \Cref{alg:three-operator-splitting} with the following estimator:
\begin{equation} \label{eqn:stochastic-estimator}
    u_t := \frac{1}{|Q_t|} \sum_{\xi \in Q_t} \nabla \tilde{f}(z_t,\xi),
\end{equation}
where $Q_t$ is a set of $|Q_t|$ \textit{i.i.d.} samples from distribution $\mathcal{P}$.

\begin{theorem}
\label{thm:three-operator-splitting-stochastic}
Consider Problem \eqref{eqn:model-problem-stochastic}. Instate the assumptions of \Cref{thm:TOS-Lipschitz}. Further, assume that the following conditions hold:\\[0.5em]
\emph{(v)}~$\nabla \tilde{f} (x,\xi)$ is an unbiased estimator of $\nabla f(x)$,
\begin{equation*}
    \mathbb{E}_\xi [\nabla \tilde{f}(x,\xi)] = \nabla f(x),\quad \forall x \in \R^n.
\end{equation*}
\emph{(vi)}~$\nabla \tilde{f} (x,\xi)$ has bounded variance: For some $\sigma < +\infty$, 
\begin{equation*}
    \mathbb{E}_\xi [\norm{\nabla \tilde{f}(x,\xi) - \nabla f(x)}^2] \leq \sigma^2, \quad \forall x \in \R^n.
\end{equation*}
Consider TOS (\Cref{alg:three-operator-splitting}) with the stochastic gradient estimator \eqref{eqn:stochastic-estimator} instead of $\nabla f(z_t)$.
Choose the algorithm parameters
\begin{align*} 
\gamma_t & = \frac{D_g}{2 (G_f + L_g + L_h) T^{2/3}} \qquad \text{and} \qquad |Q_t|  = \Big\lceil \frac{T^{2/3}}{2(G_f + L_g + L_h)^2} \Big\rceil.
\end{align*}
 Then, $z_\tau$ returned by the algorithm after $T$ iterations satisfies, $\forall x \in \dom{\phi}$,
\begin{align*}
\resizebox{0.98\hsize}{!}
{$
\displaystyle
 \mathbb{E}_\tau \mathbb{E}[ \ip{\nabla f(z_t)}{z_t - x} + g(z_t) - g(x) + h(z_t) - h(x) ] \leq D_g (G_f + L_g + L_h) \bigg( \frac{\sqrt{4 + 2\sigma^2}}{T^{1/2}} + \frac{4 +\sqrt{2}+ \sigma^2}{T^{1/3}} \bigg).
$}
\end{align*}
\end{theorem}

Similar to \Cref{thm:TOS-indicator-Lipschitz},
we can specify guarantees for the case where $g$ is an indicator function and $h$ is $L_h$-Lipschitz continuous. We skip the details.

Next, analogous to Problem~\eqref{eqn:model-problem-constraint}, we consider the nonconvex expectation minimization problem over the intersection of convex sets:
\begin{equation}
\label{eqn:model-problem-constraint-stochastic}
\min_{x \in \R^n} \quad f(x) := \mathbb{E}_\xi \tilde{f}(x,\xi) \quad \mathrm{subj.~to}
\quad x \in \mathcal{G} \cap \mathcal{H}.
\end{equation}
The next theorem presents convergence guarantees of TOS for this problem.

\begin{theorem} \label{thm:TOS-stochastic-indicator}
Consider Problem \eqref{eqn:model-problem-constraint-stochastic}. Instate the assumptions of \Cref{thm:TOS-indicator,thm:three-operator-splitting-stochastic}. 
Consider TOS (\Cref{alg:three-operator-splitting}) with the stochastic gradient estimator \eqref{eqn:stochastic-estimator} instead of $\nabla f(z_t)$.
Choose the algorithm parameters
\begin{align*} 
\gamma_t = \frac{D_\mathcal{G}}{2 G_f T^{2/3}} \qquad \text{and}\qquad |Q_t| = \Big\lceil \frac{T^{2/3}}{2G_f^2} \Big\rceil.
\end{align*}
Then, $z_\tau$ returned by the algorithm after $T$ iterations satisfies, 
\begin{align*}
\begin{aligned}
\mathbb{E}_\tau \mathbb{E} [\dist(z_\tau, \mathcal{H})] & 
\leq D_\mathcal{G} \bigg( \frac{\sqrt{4+2\sigma^2}}{T^{1/2}}  +  \frac{2 + \sqrt{2}}{T^{1/3}} \bigg),
\\[0.25em]
\mathbb{E}_\tau \mathbb{E} [\ip{\nabla f (z_\tau)}{z_\tau - x}] & 
\leq G_f D_\mathcal{G} \bigg( \frac{\sqrt{4+2\sigma^2}}{T^{1/2}} + \frac{4 + \sqrt{2} +\sigma^2}{T^{1/3}} \bigg), \quad \forall x \in \mathcal{G} \cap \mathcal{H}.
\end{aligned}
\end{align*}

\end{theorem}

\begin{corollary}
Under the assumptions listed in \Cref{thm:three-operator-splitting-stochastic} (resp. \Cref{thm:TOS-stochastic-indicator}), TOS returns an $\epsilon$-stationary point in expectation per \Cref{def:epsilon-stationary} (resp. $\epsilon$-feasible $\epsilon$-stationary point per \Cref{def:epsilon-feasible-stationary}) after $T \leq \mathcal{O}(1/\epsilon^3)$ iterations. In total, this algorithm requires drawing $\mathcal{O}(1/\epsilon^5)$ \emph{i.i.d.} samples from distribution $\mathcal{P}$. 
\end{corollary}

\begin{proof}
$\epsilon \leq \mathcal{O}(1/T^{1/3})$ implies $T \leq \mathcal{O}(1/\epsilon^3)$ iteration complexity. 
At each iteration, we use $|Q_t| = \Omega(T^{2/3})$ stochastic gradients, so the total stochastic gradients complexity is $\sum_{t = 0}^T |Q_t| = (T+1) |Q_t| = \Omega(T^{5/3}) \leq \mathcal{O}(1/\epsilon^5)$.
\end{proof}

Reducing the stochastic gradient complexity of TOS (\Cref{alg:three-operator-splitting}) via variance reduction techniques (see, for example, \citep{roux2012stochastic,johnson2013accelerating,defazio2014saga,nguyen2017sarah,fang2018spider}) can be a valuable extension. 
We leave this for a future study.

\section{Numerical Experiments}
\label{sec:experiments}

This section demonstrates the empirical performance of TOS on the quadratic assignment problem (QAP). 
QAP is a challenging formulation in the NP-hard problem class \citep{sahni1976p}. 
We focus on the \emph{relax-and-round} strategy proposed in \citep{vogelstein2015fast}. 
This strategy requires solving a nonconvex optimization problem over the Birkhoff polytope (\emph{i.e.}, the set of doubly stochastic matrices).
First, we will summarize the main steps of this \emph{relax-and-round} strategy and explain how we can use TOS in this procedure. Then, we will compare the performance of TOS against the Frank-Wolfe method (FW) \citep{frank1956algorithm,jaggi2013revisiting,lacoste2016convergence} used in  \citep{vogelstein2015fast}.

\subsection{Problem Description}

Given the cost matrices $A$ and $B \in \R^{n\times n}$, the goal in QAP is to align these matrices by finding a permutation matrix that minimizes a quadratic objective:
\begin{equation}
\label{eqn:QAP}
\begin{aligned}
& \min_{X \in \R^{n \times n}} & & \mathrm{trace}(AXB^\top X^\top ) \\
& \mathrm{subj.~to} & & X \in \{0,1\}^{n \times n}, ~~ X1_n = X^\top1_n = 1_n, 
\end{aligned}
\end{equation}
where $1_n$ denotes the $n$-dimensional vector of ones. 

The challenge comes from the combinatorial nature of the feasible region. \eqref{eqn:QAP} is NP-Hard, so \citet{vogelstein2015fast} focus on its continuous relaxation: 
\begin{align}
\label{eqn:QAP-relaxed}
\begin{aligned}
& \min_{X \in \R^{n \times n}} & & \mathrm{trace}(AXB^\top X^\top ) \\
& \mathrm{subj.~to} & & X \in [0,1]^{n \times n}, ~~ X1_n = X^\top1_n = 1_n. 
\end{aligned}
\end{align}
\eqref{eqn:QAP-relaxed} is a quadratic optimization problem over the Birkhoff polytope. Remark that the quadratic objective is nonconvex in general.  

The relax-and-round strategy of \citep{vogelstein2015fast} involves two main steps: \\
1. Finding a local optimal solution of \eqref{eqn:QAP-relaxed}. \\
2. Rounding the solution to the closest permutation matrix. 

\textbf{Solving \eqref{eqn:QAP-relaxed}.} 
Projecting an arbitrary matrix onto the Birkhoff polytope is computationally challenging and the standard algorithms in the constrained nonconvex optimization literature are inefficient for \eqref{eqn:QAP-relaxed}. 

\citet{vogelstein2015fast} employ the FW algorithm to overcome this challenge. 
FW does not require projections. Instead, at each iteration, it requires solving a linear assignment problem (LAP). The arithmetic cost of LAP by using the Hungarian method or the Jonker-Volgenant algorithm \citep{kuhn1955hungarian,munkres1957algorithms,jonker1987shortest} is $\mathcal{O}(n^3)$.

In this paper, we suggest TOS for solving \eqref{eqn:QAP-relaxed} instead of FW. To apply TOS, we can split the Birkhoff polytope in two different ways. 

One, we can consider the intersection of row-stochastic matrices and column-stochastic matrices:
\begin{align}\tag{Split~1}\label{eqn:split1}
\begin{aligned}
\mathcal{G} & = \{X \in [0,1]^{n \times n} : X1_n = 1_n\} \\
\mathcal{H} & = \{X \in [0,1]^{n \times n} : X^\top1_n = 1_n\}.
\end{aligned}
\end{align}
In this case, the projector onto $\mathcal{G}$ (resp., $\mathcal{H}$) requires projecting each row (resp., column) onto the unit simplex separately. 
The arithmetic cost of projecting each row (resp., column) is $\mathcal{O}(n)$ \citep{condat2016fast}, and we can project multiple rows (resp., columns) in parallel. 

Two, we can consider the following scheme studied in \citep{zass2006doubly,lu2016fast,pedregosa2018adaptive}:
\begin{align}\tag{Split~2}\label{eqn:split2}
\begin{aligned}
\mathcal{G} & = [0,1]^{n \times n} \\
\mathcal{H} & = \{X \in \R^{n \times n} : X1_n = X^\top1_n = 1_n\}. 
\end{aligned}
\end{align}
In this case, the projection onto $\mathcal{G}$ truncates the entries and the projection onto $\mathcal{H}$ has a closed-form solution:
\begin{align*}
\proj_{\mathcal{H}}(X)  =  X  +  \left( \frac{1}{n} I  +  \frac{1_n^\top X 1_n}{n^2} I  -  \frac{1}{n} X \right) 1_n 1_n^\top  -  \frac{1}{n} 1_n 1_n^\top X,
\end{align*}
where $I$ denotes the identity matrix. We present a derivation of this projection operator in the supplementary material. 

\textbf{Rounding.} 
The solution of \eqref{eqn:QAP-relaxed} does not immediately yield a feasible point for QAP \eqref{eqn:QAP}. We need a rounding step. 

Suppose $X_\tau$ is a solution to \eqref{eqn:QAP-relaxed}. 
A natural strategy is choosing the closest permutation matrix to $X_\tau$. We can find this permutation matrix by solving                                                 
\begin{align}
\label{eqn:rouinding-2}
\begin{aligned}
& \max_{X \in \R^{n \times n}} & & \ip{X_\tau}{X} \\
& \mathrm{subj.~to} & & X \in [0,1]^{n \times n}, ~~ X1_n = X^\top1_n = 1_n. 
\end{aligned}
\end{align}
We present the derivation of this folklore formulation in the supplementary material. 
\eqref{eqn:rouinding-2} is an instance of LAP. Hence, it can be solved in $\mathcal{O}(n^3)$ arithmetic operations via the Hungarian method or the Jonker-Volgenant algorithm.

\subsection{Numerical Results}

\textbf{Implementation details.} 
For FW, we use the exact line-search (greedy) step-size as in \citep{vogelstein2015fast}. 
For solving LAP, we employ an efficient implementation of the Hungarian method \citep{Ciao2011HungarianMethod}.

For TOS (\Cref{alg:three-operator-splitting}), we output the last iterate instead of the random variable $x_\tau$. We use $\gamma_t = 1/L_f$ step-size ($L_f$ denotes the smoothness constant of $f$) instead of the more conservative step-size that our theory suggests (which depends on $T$). $1/L_f$ is the standard rule in convex optimization, and in our experience, it works well for nonconvex problems too. 

We start both methods from the same initial point $y_1$. We construct $y_1$ by projecting a random matrix with \textit{i.i.d.} standard Gaussian entries onto the Birkhoff polytope via $1000$ iterations of the alternating projections method.

\textbf{Quality of solution.} Given a prospective solution $X_t \in \mathcal{G}$, we compute the following errors:
\begin{align}
    \text{infeasibility err.} & = \frac{\mathrm{dist}(X_t,\mathcal{H})}{\sqrt{n}} \notag \\
    \text{nonstationarity err.}  & = \frac{| \max_{X \in \mathcal{G} \cap \mathcal{H}} \ip{\nabla f(X_t)}{X_t-X} | }{\max\{f(X_t),1\}} 
\end{align}
Infeasibility error is always $0$ for FW. 
We evaluate these errors only at iterations $t = 1,2,4,8,\ldots$ to avoid extra computation. 

We evaluate the quality of the rounded solution $\tilde{X}_t$ by using the following formula:
\begin{align} \label{eqn:assignment-error}
    \text{assignment err.} & = \frac{f(\tilde{X}_t) - f(\tilde{X}_{\text{best}})}{\max \{f(\tilde{X}_{\text{best}}),1\}},
\end{align}
where $\tilde{X}_{\text{best}}$ is the best solution known for \eqref{eqn:QAP}. $\tilde{X}_{\text{best}}$ is unknown in normal practice, but it is available for the QAPLIB benchmark problems.

\textbf{Observations.}
\Cref{fig:convergence-rates} compares the empirical performance of TOS and FW for solving \eqref{eqn:QAP-relaxed} with \textsf{chr12a} and \textsf{esc128} datasets from QAPLIB. 
In particular, \emph{TOS exhibits locally linear convergence}, whereas FW converges with sublinear rates. We observed qualitatively similar behavior also with the other datasets in QAPLIB. 

Computing the gradient dominates the runtime of TOS. 
Instead, for FW, the bottleneck is solving the LAP subproblems. As a result, TOS is especially advantageous against FW when $A$ and $B$ are sparse.

Next, we examine the quality of the rounded solutions we obtain after solving \eqref{eqn:QAP-relaxed} with TOS and FW. We initialize both methods from the same point and we stop them at the same level of accuracy, when infeasibility and nonstationarity errors drop below $10^{-5}$ (recall that the infeasibility error is always 0 for FW). We round the final estimates to the closest permutation matrix and evaluate the assignment error \eqref{eqn:assignment-error}. \Cref{fig:assignment-costs} presents the results of this experiment for the 134 datasets in QAPLIB. 

Remarkably, TOS gets a \emph{better} solution on 83 problems; TOS and FW perform the same on 16; and FW outperforms TOS on 35 instances. The largest margin appears on the \textsf{chr15b} dataset where TOS scores $0.744$ lower assignment error than FW. On the other extreme, the assignment error of the FW solution is $0.253$ lower than TOS on the \textsf{chr15c} dataset. On average (over datasets), TOS outperforms FW in assignment error by a margin of $0.046$.

\textbf{Computational environment.} Experiments are performed in \textsc{Matlab} R2018a on a MacBook Pro Late 2013 with 2.6 GHz Quad-Core Intel Core i7 CPU and 16 GB 1600 MHz DDR3 memory. The source code is available online\footnote{\url{https://github.com/alpyurtsever/NonconvexTOS}}.

\textbf{Other solvers for QAP.} The literature covers numerous approaches for tackling QAP, including \emph{(i)} exact solution methods with branch-and-bound, dynamic programming, and cutting plane methods, \emph{(ii)} heuristics and metaheuristics based on local and tabu search, simulated annealing, genetic algorithms, and neural networks, and \emph{(iii)} lower bound approximation methods via spectral bounds, mixed-integer linear programming, and semidefinite programming relaxations. An extensive comparison with these methods is beyond the scope of our paper. We refer to the comprehensive survey of \citet{loiola2007survey} for more details.

\begin{figure*}[p]
\begin{center}
\includegraphics[width=0.9\textwidth]{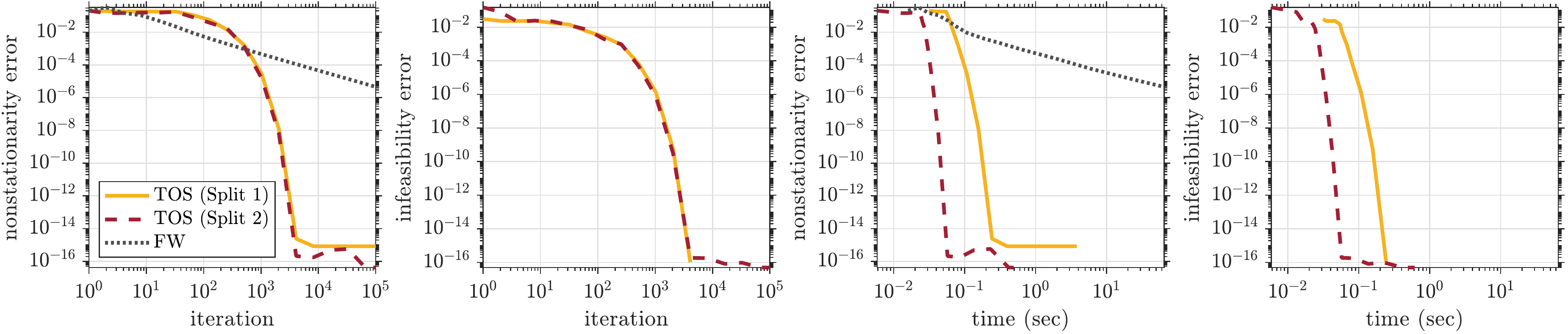}\\[1em]
\includegraphics[width=0.9\textwidth]{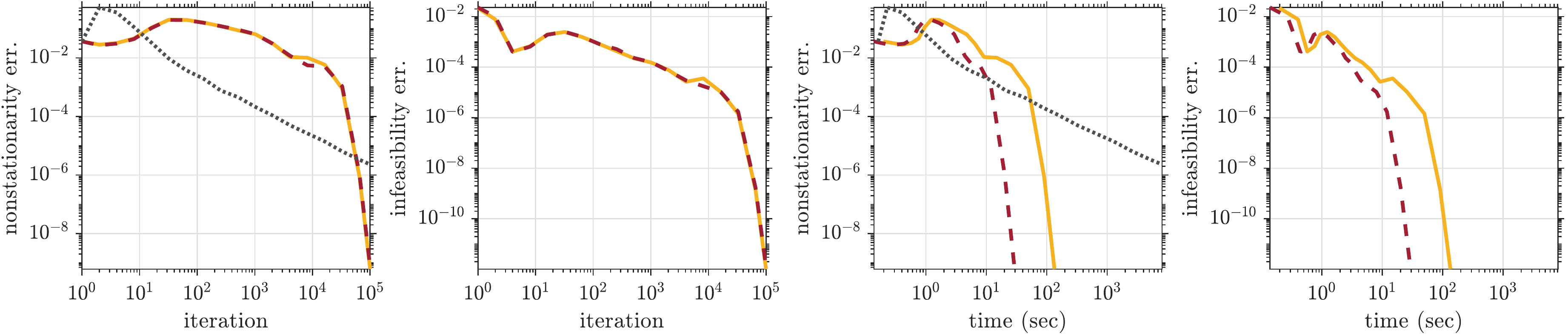}
\caption{Empirical convergence of TOS for two different formulations (\eqref{eqn:split1} and \eqref{eqn:split2}) compared against FW for solving the relaxed QAP formulation \eqref{eqn:QAP-relaxed}. The [top] row corresponds to the results for the \textsf{chr12a} dataset and the [bottom] row for the \textsf{esc128} dataset (from QAPLIB). In both cases, TOS exhibits locally linear convergence whereas FW converges sublinearly. }
\label{fig:convergence-rates}
\end{center}
\vskip 0.3in
\begin{center}
\includegraphics[width=0.9\textwidth]{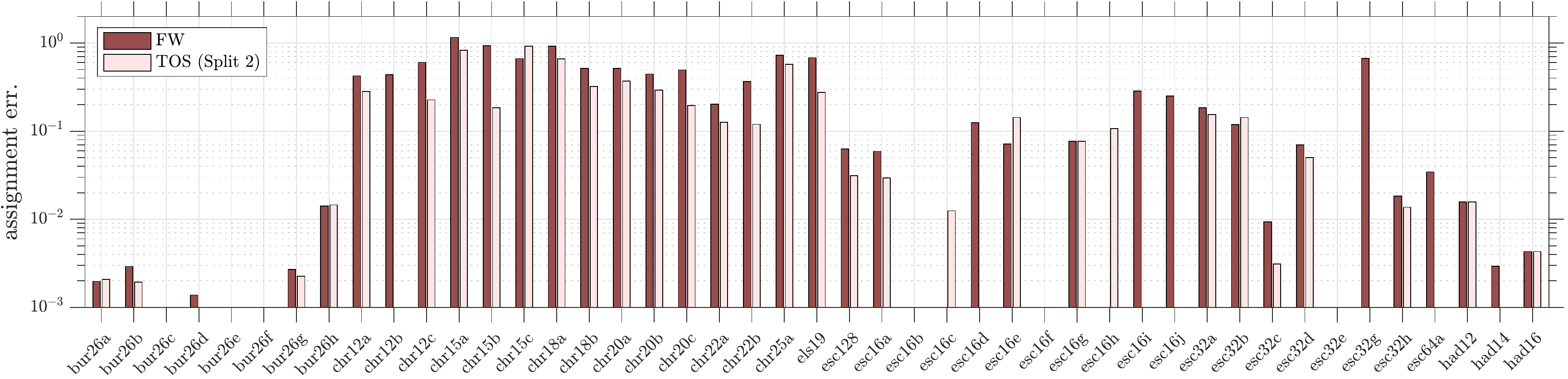}\\[0.5em]
\includegraphics[width=0.9\textwidth]{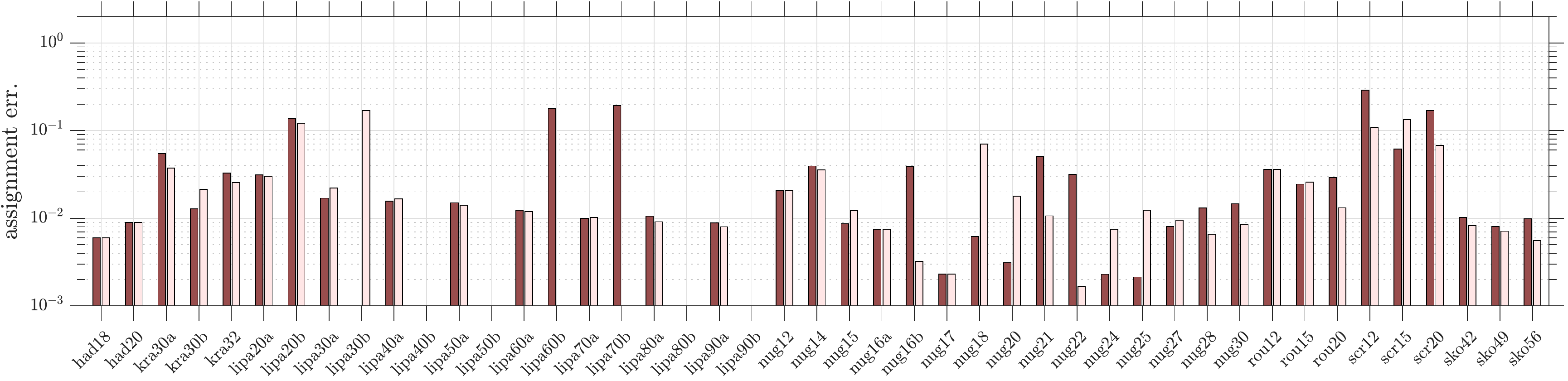}\\[0.5em]
\includegraphics[width=0.9\textwidth]{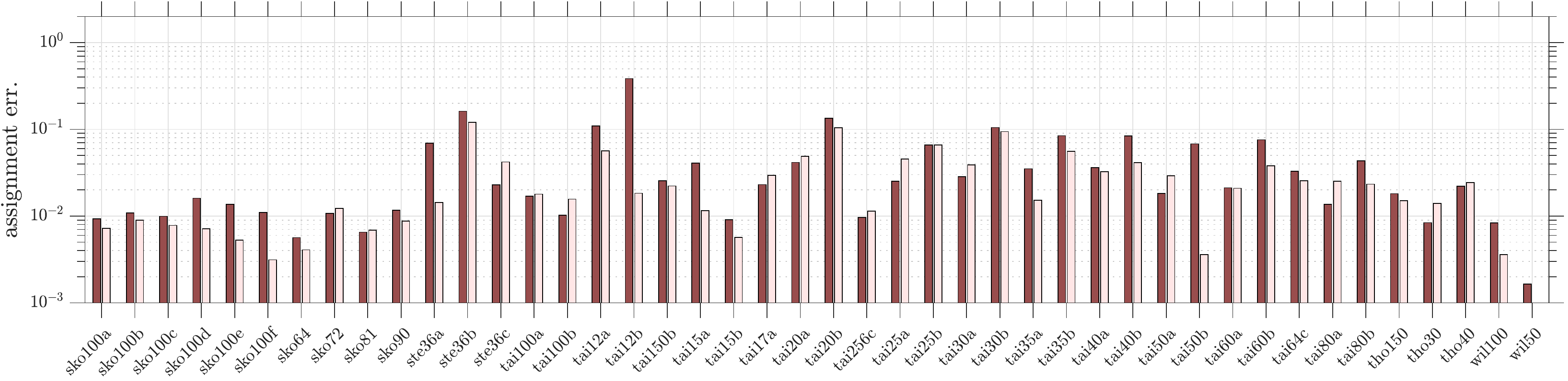}
\caption{Assignment cost (see \eqref{eqn:assignment-error}) achieved by FW and TOS with the relax-and-round strategy for solving QAP. Smaller values are better, zero means a perfect estimation. Out of 134 QAP instances in the QAPLIB library, TOS outperforms FW on 83 problems; FW is better on 35; and the two methods get the same results on 16 instances.
}
\label{fig:assignment-costs}
\end{center}
\vskip -0.3in
\end{figure*}

\section{Conclusions}
\label{sec:conclusion}

We establish the convergence guarantees of TOS for minimizing the sum of three functions, one differentiable but potentially nonconvex and two convex but potentially nonsmooth. In contrast with the existing results, our analysis permits both nonsmooth terms to be indicator functions. Moreover, we extend our analysis for stochastic problems where we have access only to an unbiased estimator of the gradient of the differentiable term.

We present numerical experiments on QAPs. The empirical performance of the proposed method is promising. In our experience, the method converges to a stationary point with locally linear rates.  

We conclude our paper with a short list of open questions and follow-up directions: \\[0.25em]
\emph{(i)}~We assume that $\dom{g}$ is bounded. This assumption is needed in our analysis since \Cref{def:epsilon-stationary} requires \eqref{eqn:epsilon-stationarity} to hold for all $x$ in $\dom{g}$. We can potentially drop this assumption by adopting a relaxed notion of stationarity where the inequality holds only on a feasible neighborhood of $\bar{z}$. Such measures are used in recent works for different problem models, \emph{e.g.}, see Definition~2.3 in \citep{nouiehed2019solving} and Definition~1 in \citep{song2020optimistic}. \\[0.25em]
\emph{(ii)}~We did not explicitly use the smoothness of the differentiable term in our analysis. One can potentially derive tighter guarantees by using the smoothness or under additional assumptions such as the Kurdyka-\L{}ojasiewicz property. \\[0.25em]
\emph{(iii)}~For the stochastic setting, we can improve the stochastic gradient complexity by using variance reduction techniques. \\[0.25em]
\emph{(iv)}~Developing an efficient implementation (that benefits from parallel computation) with an aim to investigate the full potential of TOS for solving QAP and other nonconvex problems such as the constrained and regularized neural networks is an important piece of future work.  

\begin{appendices}

\setcounter{equation}{0}
\renewcommand{\theequation}{S.\arabic{equation}}

\section{Algorithm}

We focus on the Three Operator Splitting (TOS) method \citep{davis2017three}. We initialize the method from an arbitrary $y_1 \in \dom{g}$. Then, the algorithm performs the following steps for $t = 1,2,\ldots, T$:
\begin{align}
z_{t} & = \prox_{\gamma_t g} (y_t) \label{eqn:algorithm-step-1} \\
x_{t} & = \prox_{\gamma_t h} (2z_t - y_t - \gamma_t u_t) \label{eqn:algorithm-step-2} \\
y_{t+1} & = y_t - z_t + x_t \label{eqn:algorithm-step-3}
\end{align}
After $T$ iterations, the algorithm draws $\tau$ uniformly at random from $\{1,2,\ldots,T\}$ and returns $z_\tau$. 
We use $u_t = \nabla f(z_t)$ in the deterministic setting and $u_t$ is an unbiased estimator of $\nabla f(z_t)$ in the stochastic case. 
\vspace{-0.25em}

\subsection{TOS for Minimizing the Sum of More Than Three Functions}

When the objective function consists of more than three terms, we can construct an equivalent formulation with three terms via a simple product space technique \citep{briceno2015forward}, see \Cref{sec:extensions-for-more-than-three-terms} for the details. 
An application of TOS for solving this product space reformulation is detailed in \Cref{alg:three-operator-splitting-arbitrary-terms}.

\begin{algorithm}[h]
\setstretch{1.15}
  \caption{TOS for Problem \eqref{eqn:model-problem-arbitrary-terms} via product space formulation \eqref{eqn:model-problem-product-space}}
  \label{alg:three-operator-splitting-arbitrary-terms}
\begin{algorithmic}
  \STATE {\bfseries Input:} Initial points $y_1^{(i)} \in \mathbb{R}^n$ for $i = 0,1,\ldots,m$, and a step-size sequence $\{\gamma_t\}_{t=1}^T$ \\
  \FOR{$t=1,2,\ldots,T$}
  \FOR{$i=0,1,2,\ldots,m$}
  \STATE $z_{t}^{(i)} = \prox_{\gamma_t g_i} (y_t^{(i)})$
  \ENDFOR\\
\STATE $x_{t} = \frac{1}{m+1}\Big( \sum_{i=0}^m \big( 2z_t^{(i)} - y_t^{(i)} \big) - \gamma_t \nabla f(z_t^{(0)}) \Big)$
  \FOR{$i=0,1,2,\ldots,m$}
  \STATE $y_{t+1}^{(i)} = y_t^{(i)} - z_t^{(i)} + x_t $
  \ENDFOR\\
  \ENDFOR\\
  \STATE {\bfseries Return:} Draw $\tau$ uniformly at random from $\{1,2,\ldots, T\}$ and output ${z}_\tau$.
\end{algorithmic}
\end{algorithm}

\vspace{-0.5em}

\section{Preliminaries}

\begin{lemma}[Prox-theorem] \label{lem:prox-theorem}
Let $f: \mathbb{R}^n \to \mathbb{R} \cup \{+\infty\}$ be a proper closed and convex function. Then, for any $x, u \in \mathbb{R}^n$, the followings are equivalent:
\begin{enumerate}[label=(\roman*),topsep=0pt,itemsep=0ex,partopsep=1ex,parsep=1ex]
\item $y = \prox_f(x)$.
\item $x - y \in \partial f(y)$.
\item $\ip{x-y}{z-y} \leq f(z) - f(y)$ ~for any~ $z \in \dom{f}$. \\[-0.5em]
\end{enumerate}
\end{lemma}

\begin{lemma}[Cosine rule] \label{lem:cosine-rule}
The following identity holds for any $x,y \in \mathbb{R}^n$:
\begin{align*}
\ip{x}{y} = \frac{1}{2} \norm{x+y}^2 - \frac{1}{2} \norm{x}^2 - \frac{1}{2} \norm{y}^2.
\end{align*}
\end{lemma}

\begin{lemma}
\label{lem:variance-lemma}
Suppose $f(x) = \mathbb{E}_\xi [\tilde{f} (x,\xi)]$ where $\xi$ is a random variable with distribution $\mathcal{P}$. 
Assume \\[0.25em]
\emph{(i)}~$\nabla \tilde{f} (x,\xi)$ is an unbiased estimator of $\nabla f(x)$, \emph{i.e.}, $\mathbb{E}_\xi [\nabla \tilde{f}(x,\xi)] = \nabla f(x), ~ \forall x \in \R^n,$ \\[0.25em]
\emph{(ii)}~$\nabla \tilde{f} (x,\xi)$ has bounded variance, \emph{i.e.}, there exists $\sigma < +\infty$ such that $\mathbb{E}_\xi [\norm{\nabla \tilde{f}(x,\xi) - \nabla f(x)}^2] \leq \sigma^2, ~ \forall x \in \R^n.$ \\[0.25em]
Let $u = \frac{1}{|Q|} \sum_{\xi \in Q} \nabla \tilde{f}(x,\xi)$, 
where $Q$ is a set of $|Q|$ \textit{i.i.d.} realizations from $\mathcal{P}$. 
Then, $u$ is an unbiased estimator of $\nabla f(x)$, and it satisfies the following variance bound:
\vspace{-0.25em}
\begin{equation*}
        \mathbb{E}_{Q} [\norm{u - \nabla f(x)}^2] \leq \frac{\sigma^2}{|Q|}.
\end{equation*}
\end{lemma}

\Cref{lem:prox-theorem} and \Cref{lem:cosine-rule} are elementary, we skip the proofs. 
For the proof of \Cref{lem:variance-lemma}, we refer to Lemma~2 in  \citep{yurtsever2019conditional} or \citep{reddi2016stochastic}.

\begin{lemma}[Stationary point] \label{lem:stationary-point}
$\bar{z} \in \mathrm{dom}(\phi)$ is a first-order stationary point of \eqref{eqn:model-problem} if and only if
\begin{align}
    \ip{\nabla f(\bar{z})}{\bar{z} - x} + g(\bar{z}) - g(x) + h(\bar{z}) - h(x) \leq 0, \qquad \forall  x \in \mathrm{dom}(\phi).
    \label{eqn:stationary-point-lemma}
\end{align}
\end{lemma}

\begin{proof}
Recall the definition of a stationary point: 
$\bar{z}$ is a stationary point of $f+g+h$ if  for all feasible directions $s$ at $\bar{z}$ there exist $u \in \partial g(\bar{z})$ and $v \in \partial h(\bar{z})$ such that
$\ip{\nabla f(\bar{z}) + u + v}{s} \geq 0$. Therefore, 
\begin{align}
\ip{\nabla f(\bar{z}) + u + v}{\bar{z} - x} \leq 0, \qquad \forall x \in \dom{\phi}. 
\label{eqn:stationary-point-def}
\end{align}
Since $g$ and $h$ are convex, $g(\bar{z}) - g(x) \leq \ip{u}{\bar{z} - x}$ and $h(\bar{z}) - h(x) \leq \ip{v}{\bar{z} - x}$. Hence, \eqref{eqn:stationary-point-def} implies \eqref{eqn:stationary-point-lemma}.

Next, we show the opposite direction. First, we recall the definition of the directional derivative. Given a point $\bar{z}$ and a feasible direction $s$, the directional derivative is 
\begin{align*}
g'(\bar{z},s) = \lim_{\alpha \to 0^+} \frac{g(\bar{z} + \alpha s) - g(\bar{z})}{\alpha}.
\end{align*}
Since $g$ and $h$ are convex, we have 
$g'(\bar{z},s) = \sup_{u \in \partial{g}(\bar{z})} \ip{u}{s}$ and $h'(\bar{z},s) = \sup_{v \in \partial{h}(\bar{z})} \ip{v}{s}$.

Now, suppose \eqref{eqn:stationary-point-lemma} holds. We choose $x = \bar{z} + \alpha s$ where $s$ is an arbitrary feasible direction and $\alpha > 0$, and we get
\begin{align*}
    \ip{\nabla f(\bar{z})}{\alpha s} + g(\bar{z} + \alpha s) - g(\bar{z}) + h(\bar{z} + \alpha s) - h(\bar{z}) \geq 0.
\end{align*}
We divide both sides by $\alpha$ and take $\alpha \to 0^+$,
\begin{align*}
    \ip{\nabla f(\bar{z})}{s} + \sup_{u \in \partial{g}(\bar{z})} \ip{u}{s} + \sup_{v \in \partial{h}(\bar{z})} \ip{v}{s} \geq 0.
\end{align*}
Hence, for any feasible direction $s$, there exist $u \in \partial g(\bar{z})$ and $v \in \partial h(\bar{z})$ such that
$\ip{\nabla f(\bar{z}) + u + v}{s} \geq 0$, which means $\bar{z}$ is a stationary point.
\end{proof}

\section{The Key Lemma of TOS}

The following lemma presents the common part of the analyses for the different settings that we consider in our paper. 

\begin{lemma} \label{lem:main-lemma-nonconvex-TOS}
Consider the model problem \eqref{eqn:model-problem}. Assume that the domain of $g$ is bounded by the diameter $D_g$, \emph{i.e.}, 
$\norm{x-y} \leq D_g$ for all $x, y \in \dom{g}.$
For simplicity, we suppose $y_1 \in \mathcal{G}$ (otherwise the bound explicitly depends on $y_1$). 
Then, $z_\tau$ returned by TOS after $T$ iterations with a fixed step-size $\gamma_t = \gamma$ satisfies, for all $x \in \dom{\phi}$,
\begin{align} \label{eqn:non-convex-main-bound}
\mathbb{E}_\tau [\ip{u_\tau}{x_\tau - x} + g(z_\tau) - g(x) + h(x_\tau) - h(x)] + \frac{1}{2\gamma} \mathbb{E}_\tau [ \norm{x_\tau - z_\tau}]^2
& \leq \frac{D_g^2}{2\gamma T}.
\end{align}
\end{lemma}

\begin{proof}
From the algorithm step \eqref{eqn:algorithm-step-2} and \Cref{lem:prox-theorem}, we have  
\begin{equation*}
\ip{2z_t - y_t - \gamma_t u_t - x_t}{x - x_t} \leq \gamma_t h(x) - \gamma_t h(x_t), \quad \forall x \in \dom{h}.
\end{equation*}
We rearrange this inequality as follows:
\begin{align*}
\ip{u_t}{x_t - x} + h(x_t) - h(x)
& \leq \frac{1}{\gamma_t} \ip{2z_t - y_t -x_t}{x_t - x}  \\
& = \frac{1}{\gamma_t} \ip{z_t - y_t }{z_t - x} + \frac{1}{\gamma_t} \ip{z_t - y_t }{x_t - z_t} +  \frac{1}{\gamma_t} \ip{z_t - x_t}{x_t - x}  \\
& = \frac{1}{\gamma_t} \ip{z_t - y_t }{z_t - x} + \frac{1}{\gamma_t} \ip{y_t - z_t  + x_t - x}{z_t - x_t}  \\
& = \frac{1}{\gamma_t} \ip{z_t - y_t }{z_t - x} + \frac{1}{\gamma_t} \ip{y_{t+1} - x}{y_t - y_{t+1}}, \quad \forall x \in \dom{h}.
\end{align*}
Next, we bound the first term on the right hand side by using \Cref{lem:prox-theorem} again (but this time for the update step \eqref{eqn:algorithm-step-1}):
\begin{equation*}
\ip{u_t}{x_t - x} + g(z_t) - g(x) + h(x_t) - h(x) \leq \frac{1}{\gamma_t} \ip{y_{t+1} - x}{y_t - y_{t+1}}, \quad \forall x \in \dom{\phi}.
\end{equation*}
Due to \Cref{lem:cosine-rule} and the update step \eqref{eqn:algorithm-step-3}, we get, for all $x \in \dom{\phi}$
\begin{align*}
\ip{u_t}{x_t - x} + g(z_t) - g(x) + h(x_t) - h(x)
& \leq \frac{1}{2\gamma_t} \Big( \norm{y_t - x}^2 - \norm{y_{t+1} - x}^2 - \norm{y_{t+1} - y_t}^2 \Big) \\
& = \frac{1}{2\gamma_t} \Big( \norm{y_t - x}^2 - \norm{y_{t+1} - x}^2 - \norm{x_t - z_t}^2 \Big).
\end{align*}
We use a fixed step-size $\gamma_t = \gamma > 0$, sum this inequality over $t=1,2,\ldots,T$ and divide by $T$:
\begin{align} \label{eqn:main-proof-after-telescope}
\frac{1}{T} \sum_{t=1}^T \left( \ip{u_t}{x_t - x} + g(z_t) - g(x) + h(x_t) - h(x) + \frac{1}{2\gamma} \norm{x_t - z_t}^2  \right)
& \leq \frac{1}{2\gamma T} \norm{y_1 - x}^2.
\end{align}
Note that, $\norm{y_1 - x} \leq D_g$ since both $x$ and $y_1$ are in $\dom{g}$. 
We choose $\tau$ uniformly at random from $1,2,\ldots,T$. By definition of the expectation, we get
\begin{align*}
\mathbb{E}_\tau [\ip{u_\tau}{x_\tau - x} + g(z_\tau) - g(x) + h(x_\tau) - h(x)] + \frac{1}{2\gamma} \mathbb{E}_\tau [ \norm{x_\tau - z_\tau}^2 ]
& \leq \frac{D_g^2}{2\gamma T}, \quad \forall x \in \dom{\phi}.
\end{align*}
We complete the proof by using $\mathbb{E} [ \norm{x_\tau - z_\tau} ]^2 \leq \mathbb{E} [ \norm{x_\tau - z_\tau}^2 ]$. 
\end{proof}

\clearpage

\section{Proof of \Cref{thm:TOS-Lipschitz} and \Cref{thm:TOS-indicator}}
\label{app:proof-main-thm}

\textbf{\Cref{thm:TOS-Lipschitz}.}
Consider the model problem \eqref{eqn:model-problem} under the following assumptions:\\[0.5em]
\emph{(i)}~The domain of $g$ is bounded with diameter $D_g$, \emph{i.e.}, $ \norm{x-y} \leq D_g, ~ \forall x, y \in \dom{g}$. \\[0.25em]
\emph{(ii)}~$g$ is $L_g$-Lipschitz continuous on its domain, \emph{i.e.}, $g(x) - g(y) \leq L_g \norm{x-y}, ~ \forall x, y \in \dom{g}$. \\[0.25em]
\emph{(iii)}~The gradient of $f$ is bounded by $G_f$ on the domain of $g$,  \emph{i.e.}, $\norm{\nabla f(x)} \leq G_f, ~ \forall x \in \dom{g}$. \\[0.25em]
\emph{(iv)}~$h$ is $L_h$-Lipschitz continuous on $\R^n$, \emph{i.e.}, $ h(x) - h(y) \leq L_h \norm{x-y}, ~ \forall x, y \in \R^n.$ \\[0.5em]
Choose $y_1 \in \dom{g}$. 
Then, $z_\tau$ returned by TOS after $T$ iterations with the fixed step-size $\gamma = \frac{D_g}{2 (G_f + L_g + L_h) T^{2/3}}$ satisfies
\begin{multline*}
    \mathbb{E}_{\tau}[\ip{\nabla f (z_\tau)}{z_\tau - x} + g(z_\tau) - g(x) +  h(z_\tau) - h(x)]  \\ \leq 
    D_g (G_f + L_g + L_h) \frac{1}{T^{1/3}} + D_g (G_f + L_h) \left( \frac{1}{T^{2/3}} + \frac{1}{T^{1/2}} + \frac{1}{T^{1/3}} \right), \quad  \forall x \in \dom{\phi}.
\end{multline*}

\begin{proof}
From \Cref{lem:main-lemma-nonconvex-TOS}, we know
\begin{align*} 
\mathbb{E}_\tau [\ip{\nabla f(z_\tau)}{x_\tau - x} + g(z_\tau) - g(x) + h(x_\tau) - h(x)] + \frac{1}{2\gamma} \mathbb{E}_\tau [ \norm{x_\tau - z_\tau}]^2
& \leq \frac{D_g^2}{2\gamma T}, \quad \forall x \in \dom{\phi}.
\end{align*}
Based on assumptions \emph{(iii)} and \emph{(iv)}, 
\begin{align}
\mathbb{E}_\tau[\ip{\nabla f (z_\tau)}{z_\tau - x} & + g(z_\tau) - g(x) +  h(z_\tau) -  h(x)]  + \frac{1}{2\gamma} \mathbb{E}_\tau [ \norm{x_\tau - z_\tau} ]^2 \notag \\ 
& \leq \frac{D_g^2}{2\gamma T} + \mathbb{E}_\tau [\ip{\nabla f (z_\tau)}{z_\tau - x_\tau} + h(z_\tau) - h(x_\tau)] \notag \\ 
& \leq \frac{D_g^2}{2\gamma T} + (G_f + L_h) \, \mathbb{E}_\tau [\norm{z_\tau - x_\tau}],
\quad \forall x \in \dom{\phi}. \label{eqn:proof-theorem1-step1}
\end{align}
At the same time, from assumptions \emph{(i)}, \emph{(ii)}, \emph{(iii)} and \emph{(iv)}, we have
\begin{align}
\mathbb{E}_\tau [\ip{\nabla f (z_\tau)}{z_\tau - x} + g(z_\tau) - g(x) +  h(z_\tau) - h(x)] 
& \geq -(G_f + L_g + L_h)\mathbb{E}_\tau [\norm{z_\tau - x}] \notag \\[0.5em]
& \geq -(G_f + L_g + L_h) D_g,
\quad \forall x \in \dom{\phi}. \label{eqn:proof-theorem1-step2}
\end{align}
Combining \eqref{eqn:proof-theorem1-step1} and \eqref{eqn:proof-theorem1-step2}, we obtain the following second order inequality of $\mathbb{E}_\tau[\norm{x_\tau - z_\tau}]$:
\begin{align*}
 \frac{1}{2\gamma} \mathbb{E}_\tau [ \norm{x_\tau - z_\tau} ]^2 
- (G_f + L_h) \, \mathbb{E}_\tau[\norm{z_\tau - x_\tau}] -  \frac{D_g^2}{2\gamma T} - (G_f + L_g + L_h) D_g
\leq 0.
\end{align*}
Solving this inequality leads to the following bound:
\begin{align*}
\mathbb{E}_\tau [\norm{z_\tau - x_\tau}]
\leq 2 \gamma (G_f + L_h) + \frac{D_g}{T^{1/2}} + \sqrt{2 \gamma (G_f + L_g + L_h) D_g}.
\end{align*}
After substituting the step-size $\gamma = \frac{D_g}{2 (G_f + L_g + L_h) T^{2/3}}$, this becomes
\begin{align*}
\mathbb{E}_\tau[\norm{z_\tau - x_\tau}]
\leq D_g \left( \frac{1}{T^{2/3}} + \frac{1}{T^{1/2}} + \frac{1}{T^{1/3}} \right).
\end{align*}
We complete the proof by substituting this inequality back into \eqref{eqn:proof-theorem1-step1}. 
\end{proof}

\textbf{\Cref{thm:TOS-indicator-Lipschitz}.}
Consider the model problem \eqref{eqn:model-problem} under the following assumptions:\\[0.5em]
\emph{(i)}~$g$ is the indicator function of convex closed bounded set $\mathcal{G} \subseteq \R^n$ with diameter $D_\mathcal{G} := \sup_{x,y \in \mathcal{G}} \norm{x-y}$. \\[0.25em]
\emph{(ii)}~The gradient of $f$ is bounded by $G_f$ in the domain of $g$,  \emph{i.e.}, $\norm{\nabla f(x)} \leq G_f, ~ \forall x \in \dom{g}$. \\[0.25em]
\emph{(iii)}~$h$ is $L_h$-Lipschitz continuous on $\R^n$, \emph{i.e.}, $ h(x) - h(y) \leq L_h \norm{x-y}, ~ \forall x, y \in \R^n.$ \\[0.5em]
Choose $y_1 \in \mathcal{G}$. 
Then, $z_\tau \in \mathcal{G}$ returned by TOS after $T$ iterations with the fixed step-size $\gamma = \frac{D_\mathcal{G}}{2 (G_f + L_h) T^{2/3}}$ satisfies
\begin{align*}
    \mathbb{E}_{\tau}[\ip{\nabla f (z_\tau)}{z_\tau - x} + h(z_\tau) - h(x)]  \leq D_g (G_f + L_h) \left( \frac{1}{T^{2/3}} + \frac{1}{T^{1/2}} + \frac{2}{T^{1/3}} \right), \quad  \forall x \in \mathcal{G}.
\end{align*}

\begin{proof} Corollary~1 is a direct consequence of Theorem~1 with 
 $\dom{g} = \dom{\phi} = \mathcal{G}$. The assumptions \emph{(i)} and \emph{(ii)} from Theorem~1 hold with $D_g = D_\mathcal{G}$ and $L_g = 0$. We can set $g(z_\tau) = g(x) = 0$ because both $z_\tau$ and $x$ are in $\mathcal{G}$.
\end{proof}

\textbf{\Cref{thm:TOS-indicator}.}
Consider the model problem \eqref{eqn:model-problem-constraint} under the following assumptions:\\[0.5em]
\emph{(i)}~$\mathcal{G} \subseteq \R^n$ is a bounded closed convex set with diameter $D_\mathcal{G} := \sup_{x,y \in \mathcal{G}} \norm{x-y}$. \\[0.25em]
\emph{(ii)}~The gradient of $f$ is bounded by $G_f$ on $\mathcal{G}$,  \emph{i.e.}, $\norm{\nabla f(x)} \leq G_f, ~ \forall x \in \mathcal{G}$. \\[0.25em]
\emph{(iii)}~$\mathcal{H} \subseteq \R^n$ is a closed convex set. \\[0.5em]
Choose $y_1 \in \mathcal{G}$. 
Then, $z_\tau$ returned by TOS after $T$ iterations with the fixed step-size $\gamma_t = \frac{D_\mathcal{G}}{2 G_f T^{2/3}}$ satisfies:
\begin{align*}
\mathbb{E}_\tau [\dist(z_\tau, \mathcal{H})] & \leq D_\mathcal{G} \left(  \frac{1}{T^{2/3}} + \frac{1}{T^{1/2}} + \frac{1}{T^{1/3}} \right), \\
\mathbb{E}_\tau [\ip{\nabla f (z_\tau)}{z_\tau - x}] & \leq G_f D_\mathcal{G} \left( \frac{1}{T^{2/3}} + \frac{1}{T^{1/2}}  + \frac{2}{T^{1/3}} \right), \quad \forall x \in \mathcal{G} \cap \mathcal{H}.
\end{align*}

\begin{proof}
The proof is similar to the proof of \Cref{thm:TOS-Lipschitz}. We present the details for completeness. 

Since $x_\tau \in \mathcal{H}$, $z_\tau \in \mathcal{G}$, and $x \in \mathcal{G} \cap \mathcal{H}$, we can simplify \Cref{lem:main-lemma-nonconvex-TOS} as 
\begin{align*} 
\mathbb{E}_\tau [\ip{\nabla f(z_\tau)}{x_\tau - x} ] + \frac{1}{2\gamma} \mathbb{E}_\tau [ \norm{x_\tau - z_\tau}]^2
\leq \frac{D_\mathcal{G}^2}{2\gamma T}, \quad \forall x \in \mathcal{G} \cap \mathcal{H}.
\end{align*}
Using the Cauchy-Schwartz inequality and assumption \emph{(ii)}, we get
\begin{align} \label{eqn:proof-theorem2-step1}
\mathbb{E}_\tau [\ip{\nabla f (z_\tau)}{z_\tau - x}] + \frac{1}{2\gamma} \mathbb{E} [ \norm{x_\tau - z_\tau}]^2
& \leq \frac{D_\mathcal{G}^2}{2\gamma T} +  \mathbb{E}_\tau[\ip{\nabla f (z_\tau)}{z_\tau - x_\tau}] \notag \\
& \leq \frac{D_\mathcal{G}^2}{2\gamma T} +  \mathbb{E}_\tau[\norm{\nabla f (z_\tau)} \norm{z_\tau - x_\tau}] \notag \\
& \leq \frac{D_\mathcal{G}^2}{2\gamma T} +  G_f \mathbb{E}_\tau[\norm{z_\tau - x_\tau}], \quad \forall x \in \mathcal{G} \cap \mathcal{H}.
\end{align}
At the same time, from the assumptions \emph{(i)} and \emph{(ii)}, we have
\begin{align} \label{eqn:proof-theorem2-step2}
\mathbb{E}_\tau [\ip{\nabla f (z_\tau)}{z_\tau - x}] 
\geq - \mathbb{E}_\tau [\norm{\nabla f (z_\tau)}\norm{z_\tau - x}] 
\geq - G_f D_\mathcal{G},
\quad \forall x \in \mathcal{G} \cap \mathcal{H}. 
\end{align}
Combining \eqref{eqn:proof-theorem2-step1} and \eqref{eqn:proof-theorem2-step2}, we obtain the following second order inequality of $\mathbb{E}_\tau[\norm{x_\tau - z_\tau}]$:
\begin{align*}
 \frac{1}{2\gamma} \mathbb{E}_\tau [ \norm{x_\tau - z_\tau} ]^2 
- G_f \mathbb{E}_\tau[\norm{z_\tau - x_\tau}] -  \frac{D_\mathcal{G}^2}{2\gamma T} - G_f D_\mathcal{G}
\leq 0.
\end{align*}
Solving this inequality leads to the following bound:
\begin{align*}
\mathbb{E}_\tau [\norm{z_\tau - x_\tau}] 
\leq 2 G_f \gamma + \frac{D_\mathcal{G}}{T^{1/2}} + \sqrt{2\gamma G_f D_\mathcal{G}}.
\end{align*}
After substituting the step-size $\gamma = \frac{D_\mathcal{G}}{2 G_f T^{2/3}}$, this becomes
\begin{align} \label{eqn:proof-theorem2-step3}
\mathbb{E}_\tau[\norm{z_\tau - x_\tau}]
\leq D_\mathcal{G} \left( \frac{1}{T^{2/3}} + \frac{1}{T^{1/2}} + \frac{1}{T^{1/3}} \right).
\end{align}
By definition, $\mathbb{E}_\tau [\dist(z_\tau, \mathcal{H})] = \mathbb{E}_\tau [ \inf_{x \in \mathcal{H}} \norm{z_\tau - x}]  \leq \mathbb{E}_\tau [\norm{z_\tau - x_\tau}] $. 
Hence, we obtained the desired bound on the infeasibility error. 
Finally, the bound on the nonstationarity error follows by substituting \eqref{eqn:proof-theorem2-step3} into \eqref{eqn:proof-theorem2-step1}. 
\end{proof}

\section{Proof of \Cref{thm:three-operator-splitting-stochastic} and \Cref{thm:TOS-stochastic-indicator}}

\textbf{\Cref{thm:three-operator-splitting-stochastic}.}
Consider the model problem \eqref{eqn:model-problem-stochastic} under the following assumptions:\\[0.5em]
\emph{(i)}~The domain of $g$ is bounded with diameter $D_g$, \emph{i.e.}, $ \norm{x-y} \leq D_g, ~ \forall x, y \in \dom{g}. $ \\[0.25em]
\emph{(ii)}~$g$ is $L_g$-Lipschitz continuous on its domain, \emph{i.e.}, $g(x) - g(y) \leq L_g \norm{x-y}, ~ \forall x, y \in \dom{g}. $ \\[0.25em]
\emph{(iii)}~The gradient of $f$ is bounded by $G_f$ on the domain of $g$,  \emph{i.e.}, $\norm{\nabla f(x)} \leq G_f, ~ \forall x \in \dom{g}$. \\[0.25em]
\emph{(iv)}~$h$ is $L_h$-Lipschitz continuous on $\R^n$, \emph{i.e.}, $ h(x) - h(y) \leq L_h \norm{x-y}, ~ \forall x, y \in \R^n$. \\[0.25em]
\emph{(v)}~$\nabla \tilde{f} (x,\xi)$ is an unbiased estimator of $\nabla f(x)$, \emph{i.e.}, $\mathbb{E}_\xi [\nabla \tilde{f}(x,\xi)] = \nabla f(x), ~ \forall x \in \R^n$. \\[0.25em]
\emph{(vi)}~$\nabla \tilde{f} (x,\xi)$ has bounded variance, \emph{i.e.}, there exists $\sigma < +\infty$ such that $\mathbb{E}_\xi [\norm{\nabla \tilde{f}(x,\xi) - \nabla f(x)}^2] \leq \sigma^2, ~ \forall x \in \R^n.$ \\[0.5em]
Choose $y_1 \in \dom{g}$. 
Use TOS with the stochastic gradient estimator
\begin{equation*}
    u_t := \frac{1}{|Q_t|} \sum_{\xi \in Q_t} \nabla \tilde{f}(z_t,\xi), \end{equation*}
where $Q_t$ is a set of $|Q_t|$ \textit{i.i.d.} realizations from distribution $\mathcal{P}$. 
Use the fixed mini-batch size $|Q_t| = |Q| = \Big\lceil \frac{T^{2/3}}{2(G_f + L_g + L_h)^2} \Big\rceil$, and the fixed step-size 
$\gamma_t = \gamma  = \frac{D_g}{2 (G_f + L_g + L_h) T^{2/3}}$ where $T$ is the total number of iterations. 
Then, $z_\tau$ returned by the algorithm satisfies:
\begin{multline*} 
        \mathbb{E}_\tau \mathbb{E}[ \ip{\nabla f(z_t)}{z_t - x} + g(z_t) - g(x) + h(z_t) - h(x) ] \\
\leq D_g (G_f + L_g + L_h) \frac{2 + \sigma^2}{T^{1/3}} + D_g (G_f + L_h) \left(\frac{2}{T^{2/3}} + \frac{\sqrt{4+2\sigma^2}}{T^{1/2}} + \frac{\sqrt{2}}{T^{1/3}}\right), \quad \forall x \in \dom{\phi}.
\end{multline*}

\begin{proof}
The proof follows similarly to the proof of \Cref{thm:TOS-Lipschitz} but we need to take care of the noise and variance terms. 
From \Cref{lem:main-lemma-nonconvex-TOS}, we have for all $x \in \dom{\phi}$, 
\begin{align} \label{eqn:proof-theorem3-step1}
\frac{1}{T} \sum_{t=1}^T \left( \ip{u_t}{x_t - x} + g(z_t) - g(x) + h(x_t) - h(x) + \frac{1}{2\gamma} \norm{x_t - z_t}^2  \right)
& \leq \frac{D_g^2}{2\gamma T}.
\end{align}
We first focus on the inner product term. Let us decompose it as
\begin{align*} 
\ip{u_t}{x_t - x}
= \ip{\nabla f(z_t)}{z_t - x} + \ip{\nabla f(z_t)}{x_t - z_t} + \ip{u_t - \nabla f(z_t)}{z_t - x} + \ip{u_t - \nabla f(z_t)}{x_t - z_t}.
\end{align*}
Using the Cauchy-Schwartz inequality, Young's inequality, and the assumptions \emph{(i)} and \emph{(iii)}, we get
\begin{align*} 
\ip{u_t}{x_t - x}
& \geq \ip{\nabla f(z_t)}{z_t - x} - G_f \norm{x_t - z_t} - \frac{1}{\alpha} \norm{u_t - \nabla f(z_t)}^2 - \frac{\alpha}{2} \norm{z_t - x}^2 - \frac{\alpha}{2} \norm{z_t - x_t}^2 \\
& \geq \ip{\nabla f(z_t)}{z_t - x} - G_f \norm{x_t - z_t} - \frac{1}{\alpha}\norm{u_t - \nabla f(z_t)}^2  - \frac{\alpha}{2} D_g^2 - \frac{\alpha}{2} \norm{z_t - x_t}^2, \quad \forall \alpha > 0.
\end{align*}
We take the expectation of both sides (over $Q_t$) and use \Cref{lem:variance-lemma} to get, for all $\alpha > 0$,
\begin{align} \label{eqn:proof-theorem3-step2}
\mathbb{E} [\ip{u_t}{x_t - x}]
& \geq \mathbb{E} [\ip{\nabla f(z_t)}{z_t - x}] - G_f \mathbb{E} [\norm{x_t - z_t}] - \frac{\sigma^2}{\alpha|Q|} - \frac{\alpha}{2} D_g^2 - \frac{\alpha}{2} \mathbb{E}[\norm{z_t - x_t}^2].
\end{align}
Now, we take the expectation of \eqref{eqn:proof-theorem3-step1} and substitute \eqref{eqn:proof-theorem3-step2} into this inequality. 
This leads to 
\begin{align} \label{eqn:proof-theorem3-theorem4-anchor}
        \frac{1}{T} \sum_{t=1}^T \mathbb{E} \Big[ \ip{\nabla f(z_t)}{z_t - x} + g(z_t) - g(x) & + h(x_t) - h(x)  + \Big(\frac{1}{2\gamma} - \frac{\alpha}{2} \Big) \norm{x_t - z_t}^2 - G_f \norm{x_t - z_t} \Big] \notag \\
 & \leq \frac{D_g^2}{2\gamma T} + \frac{\alpha D_g^2}{2} + \frac{\sigma^2}{\alpha |Q| }, \quad \forall x \in \dom{\phi}, ~ \forall \alpha > 0.
\end{align}
Based on the assumption \emph{(iv)}, we obtain 
\begin{multline} \label{eqn:stochastic-summation}
        \frac{1}{T} \sum_{t=1}^T \mathbb{E} \Big[ \ip{\nabla f(z_t)}{z_t - x} + g(z_t) - g(x) + h(z_t) - h(x) + \Big(\frac{1}{2\gamma} - \frac{\alpha}{2} \Big) \norm{x_t \!-\! z_t}^2 - (G_f + L_h) \norm{x_t \! - \! z_t} \Big] \\
 \leq \frac{D_g^2}{2\gamma T} + \frac{\alpha D_g^2}{2} + \frac{\sigma^2}{\alpha |Q| }, \quad \forall x \in \dom{\phi}, ~  \forall \alpha > 0.
\end{multline}
At the same time, based on the assumptions \emph{(i)}, \emph{(ii)}, \emph{(iii)}, and \emph{(iv)}, we have
\begin{equation} \label{eqn:proof-theorem3-step}
\ip{\nabla f(z_t)}{z_t - x} + g(z_t) - g(x) + h(z_t) - h(x) \geq - (G_f + L_g + L_h) D_g,  \quad \forall x \in \dom{\phi}. 
\end{equation}
Substituting \eqref{eqn:proof-theorem3-step} back into \eqref{eqn:stochastic-summation} leads to 
\begin{align*}
        \frac{1}{T} \sum_{t=1}^T \mathbb{E} \Big[ \Big(\frac{1}{2\gamma} - \frac{\alpha}{2} \Big) \norm{x_t - z_t}^2 - (G_f + L_h) \norm{x_t - z_t} \Big] 
 \leq \frac{D_g^2}{2\gamma T} + \frac{\alpha D_g^2}{2} + \frac{\sigma^2}{\alpha |Q|} + (G_f + L_g + L_h) D_g,
\end{align*}
for all $\alpha > 0$. We choose $\tau$ uniformly random over $1,2,\ldots, T$. Hence, 
by definition of the expectation over $\tau$, we have $\forall \alpha > 0$,
\begin{align*}
        \Big(\frac{1}{2\gamma} - \frac{\alpha}{2} \Big) \mathbb{E}_\tau \mathbb{E}[\norm{x_\tau \!-\! z_\tau}^2] - (G_f \!+\! L_h) \mathbb{E}_\tau \mathbb{E}[ \norm{x_\tau \!-\! z_\tau} ] 
 \leq \frac{D_g^2}{2\gamma T} + \frac{\alpha D_g^2}{2} + \frac{\sigma^2}{\alpha |Q|} + (G_f \!+\! L_g \!+\! L_h) D_g.
\end{align*}
Choose $\alpha$ such that $\alpha  \leq \frac{1}{\gamma}$. Note that, $\mathbb{E}_\tau \mathbb{E}[\norm{x_\tau - z_\tau}]^2 \leq \mathbb{E}_\tau \mathbb{E}[\norm{x_\tau - z_\tau}^2]$. Therefore, 
\begin{align*}
        \Big(\frac{1}{2\gamma} - \frac{\alpha}{2} \Big) \mathbb{E}_\tau \mathbb{E}[\norm{x_\tau \!-\! z_\tau}]^2 - (G_f \!+\! L_h) \mathbb{E}_\tau \mathbb{E}[ \norm{x_\tau \!-\! z_\tau} ] 
 \leq \frac{D_g^2}{2\gamma T} + \frac{\alpha D_g^2}{2} + \frac{\sigma^2}{\alpha |Q|} + (G_f \!+\! L_g \!+\! L_h) D_g.
\end{align*}
We choose $\gamma = \frac{D_g}{2 (G_f + L_g + L_h) T^{2/3}}$, $\alpha =\frac{2(G_f + L_g + L_h)}{D_g T^{1/3}}$, and $|Q| = \lceil \frac{T^{2/3}}{2(G_f + L_g + L_h)^2} \rceil$. This leads to
\begin{equation*}
       \Big( \frac{T-1}{T} \Big) \mathbb{E}_\tau \mathbb{E}[\norm{x_\tau - z_\tau}]^2 - \frac{D_g}{T^{2/3}} \mathbb{E}_\tau \mathbb{E}[ \norm{x_\tau - z_\tau} ] 
 \leq D_g^2 \left( \frac{2}{T} + \frac{\sigma^2}{T} +  \frac{1}{T^{2/3}} \right).
\end{equation*}
By solving this inequality, we obtain, for any $T \geq2$, 
\begin{equation}\label{eqn:proof-theorem3-step3}
\mathbb{E}_\tau \mathbb{E}[\norm{x_\tau - z_\tau}] 
\leq D_g \left(\frac{T^{1/3}}{T-1} + \sqrt{\frac{2+\sigma^2}{T-1} + \frac{T^{1/3}}{T-1}}\right)
\leq D_g \left(\frac{2}{T^{2/3}} + \frac{\sqrt{4+2\sigma^2}}{T^{1/2}} + \frac{\sqrt{2}}{T^{1/3}}\right).
\end{equation}
Combining this bound back again with \eqref{eqn:stochastic-summation}, we get
\begin{align*} 
        \mathbb{E}_\tau \mathbb{E}[ \ip{\nabla f(z_t)}{z_t - x} \!+\! g(z_t) \!-\! g(x) \!+\! h(z_t) \!-\! h(x) ] 
 \leq \frac{D_g^2}{2\gamma T} \!+\! \frac{\alpha D_g^2}{2} \!+\! \frac{\sigma^2}{\alpha |Q| } + (G_f + L_h) \mathbb{E}_\tau \mathbb{E} [\norm{x_t - z_t} ] \\
\leq D_g (G_f + L_g + L_h) \frac{2 + \sigma^2}{T^{1/3}} + D_g (G_f + L_h) \left(\frac{2}{T^{2/3}} + \frac{\sqrt{4+2\sigma^2}}{T^{1/2}} + \frac{\sqrt{2}}{T^{1/3}}\right), \quad \forall x \in \dom{\phi}.
\end{align*}
This completes the proof. 
\end{proof}

\textbf{\Cref{thm:TOS-stochastic-indicator}.}
Consider the model problem \eqref{eqn:model-problem-constraint-stochastic} under the following assumptions: \\[0.5em]
\emph{(i)}~$\mathcal{G} \subseteq \R^n$ is a bounded closed convex set with diameter $D_\mathcal{G} := \sup_{x,y \in \mathcal{G}} \norm{x-y}$. \\[0.25em]
\emph{(ii)}~The gradient of $f$ is bounded by $G_f$ on $\mathcal{G}$,  \emph{i.e.}, $\norm{\nabla f(x)} \leq G_f, ~ \forall x \in \mathcal{G}$. \\[0.25em]
\emph{(iii)}~$\mathcal{H} \subseteq \R^n$ is a closed convex set. \\[0.25em]
\emph{(iv)}~$\nabla \tilde{f} (x,\xi)$ is an unbiased estimator of $\nabla f(x)$, \emph{i.e.}, $\mathbb{E}_\xi [\nabla \tilde{f}(x,\xi)] = \nabla f(x), ~ \forall x \in \R^n$. \\[0.25em]
\emph{(v)}~$\nabla \tilde{f} (x,\xi)$ has bounded variance, \emph{i.e.}, there exists $\sigma < +\infty$ such that $\mathbb{E}_\xi [\norm{\nabla \tilde{f}(x,\xi) - \nabla f(x)}^2] \leq \sigma^2, ~ \forall x \in \R^n$. \\[0.5em]
Choose $y_1 \in \mathcal{G}$. 
Use TOS with the stochastic gradient estimator
\begin{equation*}
    u_t := \frac{1}{|Q_t|} \sum_{\xi \in Q_t} \nabla \tilde{f}(z_t,\xi), \end{equation*}
where $Q_t$ is a set of $|Q_t|$ \textit{i.i.d.} realizations from distribution $\mathcal{P}$. 
Use the fixed mini-batch size $|Q_t| = |Q| = \lceil \frac{T^{2/3}}{2 G_f^2} \rceil$, and the fixed step-size 
$\gamma_t = \gamma  = \frac{D_\mathcal{G}}{2 G_f T^{2/3}}$ where $T$ is the total number of iterations. 
Then, $z_\tau$ returned by the algorithm satisfies:
\begin{align*}
\mathbb{E}_\tau \mathbb{E} [\dist(z_\tau, \mathcal{H})] & \leq  D_\mathcal{G} \left( \frac{2}{T^{2/3}} + \frac{\sqrt{4+2\sigma^2}}{T^{1/2}} + \frac{\sqrt{2}}{T^{1/3}} \right). \\[0.25em]
\mathbb{E}_\tau \mathbb{E} [\ip{\nabla f (z_\tau)}{z_\tau - x}] 
& \leq G_f D_\mathcal{G} \left( \frac{2}{T^{2/3}} + \frac{\sqrt{4+2\sigma^2}}{T^{1/2}} + \frac{2 + \sqrt{2} + \sigma^2}{T^{1/3}} \right), \quad \forall x \in \mathcal{G} \cap \mathcal{H}.
\end{align*}

\begin{proof}
The proof of \Cref{thm:TOS-stochastic-indicator}  follows similarly to the proof of \Cref{thm:three-operator-splitting-stochastic} until \eqref{eqn:proof-theorem3-theorem4-anchor}.
Since $x_\tau \in \mathcal{H}$, $z_\tau \in \mathcal{G}$, and $x \in \mathcal{G} \cap \mathcal{H}$, we can simplify  \eqref{eqn:proof-theorem3-theorem4-anchor} as, $\forall x \in \mathcal{G} \cap \mathcal{H}, ~ \forall \alpha > 0$,
\begin{align} \label{eqn:theorem4-step1}
        \frac{1}{T} \sum_{t=1}^T \mathbb{E} \Big[ \ip{\nabla f(z_t)}{z_t - x} + \Big(\frac{1}{2\gamma} - \frac{\alpha}{2} \Big) \norm{x_t - z_t}^2 - G_f \norm{x_t - z_t} \Big] 
 \leq \frac{D_\mathcal{G}^2}{2\gamma T} + \frac{\alpha D_\mathcal{G}^2}{2} + \frac{\sigma^2}{\alpha |Q| }.
\end{align}
From the assumptions \emph{(i)} and \emph{(ii)}, we have $\ip{\nabla f (z_t)}{z_t - x} 
\geq - G_f D_\mathcal{G}$.
By definition of the expectation over $\tau$, we get,
\begin{align*}
        \Big(\frac{1}{2\gamma} - \frac{\alpha}{2} \Big) \mathbb{E}_\tau \mathbb{E}[\norm{x_\tau - z_\tau}^2] - G_f \mathbb{E}_\tau \mathbb{E}[ \norm{x_\tau - z_\tau} ] 
 \leq \frac{D_\mathcal{G}^2}{2\gamma T} + \frac{\alpha D_\mathcal{G}^2}{2} + \frac{\sigma^2}{\alpha |Q|} + G_f D_\mathcal{G}, \quad \forall \alpha > 0.
\end{align*}
By substituting $|Q| = \lceil \frac{T^{2/3}}{2 G_f^2} \rceil$ and $ = \gamma  = \frac{D_\mathcal{G}}{2 G_f T^{2/3}}$, and choosing $\alpha =\frac{2G_f}{D_\mathcal{G} T^{1/3}}$, we get
\begin{equation*}
          \Big( \frac{T-1}{T} \Big) \mathbb{E}_\tau \mathbb{E}[\norm{x_\tau - z_\tau}]^2 - \frac{D_\mathcal{G}}{T^{2/3}} \mathbb{E}_\tau \mathbb{E}[ \norm{x_\tau - z_\tau} ] 
 \leq D_\mathcal{G}^2 \left( \frac{2}{T} + \frac{\sigma^2}{T} +  \frac{1}{T^{2/3}} \right).
\end{equation*}
Solving this inequality for $\mathbb{E}_\tau \mathbb{E}[\norm{x_\tau - z_\tau}]$ yields
\begin{equation} \label{eqn:theorem4-infeasibility}
        \mathbb{E}_\tau \mathbb{E}[\norm{x_\tau - z_\tau}]
\leq D_\mathcal{G} \left(\frac{T^{1/3}}{T-1} + \sqrt{\frac{2+\sigma^2}{T-1} + \frac{T^{1/3}}{T-1}}\right)
\leq D_\mathcal{G} \left(\frac{2}{T^{2/3}} + \frac{\sqrt{4+2\sigma^2}}{T^{1/2}} + \frac{\sqrt{2}}{T^{1/3}}\right).
\end{equation}
By definition, $\mathbb{E}_\tau \mathbb{E} [\dist(z_\tau, \mathcal{H})] \leq \mathbb{E}_\tau \mathbb{E}[\norm{x_\tau - z_\tau}]$.

To obtain the desired bound on the nonstationarity error, we substitute \eqref{eqn:theorem4-infeasibility} into \eqref{eqn:theorem4-step1}:
\begin{align*} 
        \mathbb{E}_\tau \mathbb{E} [ \ip{\nabla f(z_\tau)}{z_\tau - x} ] 
& \leq \frac{D_\mathcal{G}^2}{2\gamma T} + \frac{\alpha D_\mathcal{G}^2}{2} + \frac{\sigma^2}{\alpha |Q| } + G_f \mathbb{E}_\tau \mathbb{E} [\norm{x_t - z_t}], \\[0.5em]
& \leq G_f D_\mathcal{G} \left( \frac{2}{T^{2/3}} + \frac{\sqrt{4+2\sigma^2}}{T^{1/2}} + \frac{2 + \sqrt{2} + \sigma^2}{T^{1/3}} \right), \quad \forall x \in \mathcal{G} \cap \mathcal{H}.
\end{align*}
This completes the proof.
\end{proof}

\section{Additional Details on the Numerical Experiments}

This section presents more details about the QAP experiments in \Cref{sec:experiments}. 

\subsection{On the Projection onto $\mathcal{H}$ for \eqref{eqn:split2}}

We consider the projection onto $\mathcal{H}:= \{ X \in \R^{n \times n} : X 1_n = 1_n, X^\top 1_n = 1_n \}$. 
Below, we present the derivation of the closed form formula for this projection from \citep{zass2006doubly,lu2016fast}. 

We can formulate this problem as
 \begin{align}
 \label{eqn:projection-to-x1nxt1n1n}
 \begin{aligned}
 & \min_{X \in \R^{n \times n}} & & \frac{1}{2} \norm{X_\tau - X}_F^2 \\
 & \mathrm{subj.~to} & & X1_n = 1_n, \quad X^\top1_n = 1_n.
 \end{aligned}
 \end{align}
We use the Karush-Kuhn-Tucker (KKT) optimality conditions. 
We can write the Lagrangian of \eqref{eqn:projection-to-x1nxt1n1n} as
\begin{align*}
\mathcal{L}(X,\mu_1, \mu_2) = \frac{1}{2} \norm{X_\tau - X}_F^2 + \ip{\mu_1}{X1_n - 1_n} + \ip{\mu_2}{X^\top1_n - 1_n}.
\end{align*}
Then, the Lagrangian stationarity condition yields
\begin{align} \label{eqn:lagrangian-derivative}
\nabla_X \mathcal{L}(X,\mu_1, \mu_2) = X - X_\tau + \mu_1 1_n^\top + 1_n \mu_2^\top .
\end{align}
Our goal is to find a triplet $(X, \mu_1, \mu_2) \in \R^{n \times n} \times \R^n \times \R^n$ that satisfies
\begin{equation}\label{eqn:kkt}
\nabla_X \mathcal{L}(X,\mu_1, \mu_2) = 0, \quad  X 1_n = 1_n ,  \quad  \text{and} \quad  X^\top 1_n  = 1_n .
\end{equation}
Based on \eqref{eqn:lagrangian-derivative}, the first condition (Lagrangian stationarity) leads to
\begin{align} \label{eqn:lagrangian-stationarity} 
X_\tau  = X + \mu_1 1_n^\top + 1_n \mu_2^\top.
\end{align}
We multiply this by $1_n$ from the right, and by $1_n^\top$ from the left:
\begin{align} 
X_\tau 1_n  
 = X 1_n + \mu_1 1_n^\top 1_n + 1_n \mu_2^\top 1_n = 1_n + n \mu_1 + 1_n 1_n^\top \mu_2, \label{eqn:kkt-2}
\end{align}
where we also used the feasibility condition in \eqref{eqn:kkt}. 

Let us try to find a solution with $\mu_1 = \mu_2 = \mu$. Under this assumption, \eqref{eqn:kkt-2} becomes
\begin{align} \label{eqn:kkt-mu-1}
X_\tau 1_n  
 = 1_n + (nI + 1_n 1_n^\top) \mu 
 \quad \iff \quad 
 \mu = (nI + 1_n 1_n^\top) ^{-1} (X_\tau- I) 1_n.
\end{align}
We can do the inversion by using the Sherman–Morrison formula:
\begin{align} \label{eqn:sherman-morrison}
 (nI + 1_n 1_n^\top) ^{-1} = \frac{1}{n} I - \frac{1}{2n^2} 1_n 1_n^\top.
\end{align}
By substituting \eqref{eqn:sherman-morrison} back into \eqref{eqn:kkt-mu-1}, we get
\begin{align} \label{eqn:kkt-mu-2} 
 \mu  = \left( \frac{1}{n} I - \frac{1}{2n^2} 1_n 1_n^\top \right) (X_\tau- I) 1_n  
 = \frac{1}{n} X_\tau 1_n - \frac{1}{2n^2} 1_n 1_n^\top X_\tau 1_n - \frac{1}{2n} 1_n.
\end{align}
Finally, we get $X$ by replacing \eqref{eqn:kkt-mu-2} into \eqref{eqn:lagrangian-stationarity} and rearranging the order:
\begin{align*} 
X   & = X_\tau - \left(\frac{1}{n} X_\tau 1_n - \frac{1}{2n^2} 1_n 1_n^\top X_\tau 1_n - \frac{1}{2n} 1_n\right) 1_n^\top - 1_n \left(\frac{1}{n} X_\tau 1_n - \frac{1}{2n^2} 1_n 1_n^\top X_\tau 1_n - \frac{1}{2n} 1_n\right)^\top  \\
& = X_\tau - \frac{1}{n} X_\tau 1_n 1_n^\top + \frac{1_n^\top X_\tau 1_n}{n^2} 1_n  1_n^\top  - \frac{1}{n}1_n1_n^\top X_\tau^\top  +  \frac{1}{n} 1_n 1_n^\top  .
\end{align*}
$(X,\mu,\mu)$ satisfy the KKT conditions in \eqref{eqn:kkt}, hence $X$ is a solution to \eqref{eqn:projection-to-x1nxt1n1n}.

\subsection{On the Rounding Procedure for QAP: Projection onto the Set of Permutation Matrices}

Suppose $X_\tau$ is a solution to \eqref{eqn:QAP-relaxed}. 
A natural rounding strategy is to find the nearest permutation matrix, \emph{i.e.}, projecting $X_\tau$ onto the set of permutation matrices. 
We can do this by solving a linear assignment problem (LAP). 
Below, we present a proof for this classical result. 

The projection onto the set of permutation matrices can be formulated as
 \begin{align}
 \label{eqn:projection-to-permutation}
 \begin{aligned}
 & \min_{X \in \R^{n \times n}} & & \frac{1}{2} \norm{X_\tau - X}_F^2 \\
 & \mathrm{subj.~to} & & X \in \{0,1\}^{n \times n}, ~~ X1_n = X^\top1_n = 1_n, 
 \end{aligned}
 \end{align}
where $I$ denotes the $n \times n$ identity matrix and $1_n$ is the $n$-dimensional vector of ones. 

Recall that a permutation matrix has precisely one entry equal to $1$ in each row and column, and the other entries are $0$. 
Hence, $\norm{X}_F= n$ for all permutation matrices. Then, 
 \begin{align*}
 \norm{X_\tau - X}_F^2
 & = \norm{X_\tau}_F^2 - 2 \ip{X_\tau}{X} + \norm{X}_F^2 = \norm{X_\tau}_F^2 - 2 \ip{X_\tau}{X} + n^2.
 \end{align*}
 Therefore, \eqref{eqn:projection-to-permutation} is equivalent (in the sense that their solution sets are the same) to the following problem:
 \begin{align*}
 \begin{aligned}
 & \max_{X \in \R^{n \times n}} & & \ip{X_\tau}{X} \\
 & \mathrm{subj.~to} & & X \in \{0,1\}^{n \times n}, ~~ X1_n = X^\top1_n = 1_n.
 \end{aligned}
 \end{align*}
 Since the objective is linear, convex hull relaxation of the domain does not change the solution. Hence, the following problem is also equivalent:
\begin{align*}
\begin{aligned}
& \max_{X \in \R^{n \times n}} & & \ip{X_\tau}{X} \\
& \mathrm{subj.~to} & & X \in [0,1]^{n \times n}, ~~ X1_n = X^\top1_n = 1_n. 
\end{aligned}
\end{align*}
This is an instance of LAP. 
It can be solved in $\mathcal{O}(n^3)$ time via the the Hungarian method \citep{kuhn1955hungarian,munkres1957algorithms} or the Jonker-Volgenant algorithm \citep{jonker1987shortest}. 

\end{appendices}

\section*{Acknowledgements}

AY received support from the Early Postdoc.Mobility Fellowship P2ELP2\_187955 from Swiss National Science Foundation and partial postdoctoral support from NSF-CAREER grant IIS-1846088. SS acknowledges support from an NSF BIGDATA grant (1741341) and an NSF CAREER grant (1846088). 

\bibliography{ScalableQAP}
\bibliographystyle{icml2021}

\end{document}